\documentclass[12pt]{iopart}

\usepackage{amssymb}
\usepackage{amsthm}
\usepackage{esdiff}
\usepackage{color}
\usepackage{graphicx}
\usepackage{subfigure}
\usepackage{subeqn}
\usepackage{iopams}

\theoremstyle{plain}
\newtheorem{thm}{Theorem}[section]

\newtheorem{appxlem}{Lemma}[section]
\newtheorem{corollary}[thm]{Corollary}

\newtheorem{assumptions}[thm]{Assumptions}
\theoremstyle{remark}
\newtheorem{remark}[thm]{Remark}

\newcommand{\lvert}{|}
\newcommand{\rvert}{|}
\newcommand{\mod}{\mbox{mod}}

\begin{document}

\title[Filtering for linear wave equations]
{Kalman filtering and smoothing for linear wave equations with model error}

\author{Wonjung Lee$^{1,2}$, D. McDougall$^1$ and A.M. Stuart$^1$}

\address{$^1$ Mathematics Institute, University of Warwick, Coventry, CV4 7AL, UK}
\address{$^2$ Mathematical Institute, University of Oxford, Oxford, OX1 3LB, UK}

\ead{
\\
Wonjung.Lee@maths.ox.ac.uk\\
D.McDougall@warwick.ac.uk\\
A.M.Stuart@warwick.ac.uk \quad http://www.warwick.ac.uk/~masdr/
}

\begin{abstract}
Filtering is a widely used methodology for the incorporation
of observed data into time-evolving systems. It
provides an online approach to state estimation inverse
problems when data is acquired sequentially. The Kalman filter
plays a central role in many applications because it
is exact for linear systems subject to Gaussian noise,
and because it forms the basis for many approximate filters
which are used in high dimensional systems. The aim of this
paper is to study the effect of model error on the Kalman
filter, in the context of linear wave propagation problems. 
A consistency result is proved when no model error is
present, showing recovery of the true signal in the
large data limit. This result, however, is not robust: it
is also proved that arbitrarily small
model error can lead to inconsistent recovery
of the signal in the large data limit.
If the model error is in the form of a constant shift to the
velocity, the filtering and smoothing
distributions only recover a partial Fourier expansion, 
a phenomenon related to aliasing. On the other hand, for a class
of wave velocity model errors which are time-dependent, it
is possible to recover the filtering distribution exactly,
but not the smoothing distribution.
Numerical results are presented which corroborate the theory, 
and also to propose a computational approach which overcomes
the inconsistency in the presence of model error, by
relaxing the model.
\end{abstract}

\submitto{\IP}

\section{Introduction}
Filtering is a methodology for the incorporation of
data into time-evolving systems \cite{AndersonMoore79,PFLit01}. 
It provides an online approach to state estimation inverse
problems when data is acquired sequentially. 
In its most general form the dynamics and/or observing
system are subject to noise and the objective is to
compute the probability distribution of the current state,
given observations up to the current time, in a sequential
fashion.  The Kalman filter \cite{Kalman60} carries out this 
process exactly for linear dynamical systems subject to
additive Gaussian noise. A key aspect of filtering in many 
applications is to understand the effect of model error --
the mismatch between model used to filter and the source
of the data itself. 
In this paper we undertake a study
of the effect of model error on the Kalman filter in the
context of linear wave problems.

Section \ref{Kalman} is devoted to describing the linear wave
problem of interest, and deriving the Kalman filter
for it. The iterative formulae for the mean and
covariance are solved and the equivalence (as measures)
of the filtering distribution at different times is
studied. In section \ref{MCTI} we study consistency of 
the filter, examining its behaviour in the large time limit 
as more and more observations are accumulated, at points 
which are equally spaced in time. It is shown that, in the 
absence of model error, the filtering distribution on the
current state converges to a Dirac measure on the truth.
However, for the linear advection equation, it is also shown 
that arbitrarily small model error, in the form of a shift
to the wave velocity, destroys this property: the filtering distribution
converges to a Dirac measure, but it is not centred on the 
truth.
Thus the order of two operations,
namely the successive incorporation of data and the 
limit of vanishing model error,
cannot be switched; this means, practically, that small
model error can induce order one errors in filters, 
even in the presence of large amounts of data. 
All of the results in section \ref{MCTI} apply to the
smoothing distribution on the initial condition, as well
as the filtering distribution on the current state.
Section \ref{tdme} concerns non-autonomous systems, and
the effect of model error. We study the linear advection 
equation in two dimensions with time-varying
wave velocity.
Two forms of model error are studied: an error in the
wave velocity which is integrable in time, and a white
noise error. In the first case it is shown that the
filtering distribution converges to a Dirac measure on the
truth, whilst the smoothing distribution converges to
a Dirac measure which is in error, i.e., not centred on the truth.
In the second, white
noise, case both the filter and smoother converge to
a Dirac measure which is in error. In section \ref{BIPsection}
we present numerical results which illustrate the theoretical
results. We also describe a numerical approach which overcomes
the effect of model error by relaxing the model to allow the
wave velocity to be learnt from the data.

We conclude the introduction with a brief review of
the literature in this area. Filtering in high dimensional
systems is important in a range of applications, especially
within the atmospheric and geophysical sciences
\cite{bennett2002inverse,EnKFEvensen94,kal03,van2009particle}.
Recent theoretical studies have shown how the Kalman filter
can not only systematically incorporate data into a model,
but also stabilize model error arising from an unstable
numerical discretization \cite{MG07},
from an inadequate choice of parameter \cite{smith2009variational,smith2010hybrid},
and the effect
of physical model error 
in a particular application is discussed
in \cite{GN00}.
The Kalman filter is used as the basis for a number
of approximate filters which are employed in
nonlinear and non-Gaussian problems, and the
ensemble Kalman filter in particular is widely used in this 
context \cite{EnKFEvensen94,Evensen06,EnKFEvensenVanLeeuwen98}.
Although robust and widely useable, the ensemble
Kalman filter does not provably reproduce the true
distribution of the signal in the large ensemble limit, 
given data, except in the Gaussian case. 
For this reason it would be desirable to
use the particle filter \cite{PFLit01} on highly 
non-Gaussian systems. However, recent theoretical work
and a range of computational experience shows that, in its
current form, particle filters will not work well in high
dimensions \cite{SBBA08,bengtsson2008curse,bickel2008sharp}.
As a consequence a great deal of research activity
is aimed at the development of various approximation
schemes within the filtering context; see
\cite{CK04,CT09a,CT09b,HM08,CHM,MHG10} for example.
The subject of consistency of Bayesian estimators for
noisily observed systems, which forms the core of our theoretical
work in this paper, is an active area of research. In the
infinite dimensional setting, much of this work is concerned 
with linear Gaussian systems, as we are here, but is
primarily aimed at the perfect model scenario
\cite{FS09,NP08,vVvZ08}. The issue of model error
arising from spatial discretization, when filtering linear PDEs,
is studied in \cite{P06}. The idea of relaxing the model
and learning parameters in order to obtain a better
fit to the data, considered in our numerical studies,
is widely used in filtering (see the chapter by
K\"{u}nsch in \cite{PFLit01}, and the paper \cite{SDBNS09}
for an application in data assimilation).

\section{Kalman Filter on Function Space}
\label{Kalman}
\subsection{Statistical model for discrete observations}
The test model, which we propose here, is a class of PDEs
\begin{eqnarray}
\label{eq:generalmodel}
\partial_t v+\mathcal{L} v =0,
\quad \forall (x,t) \in \mathbf{T}^2 \times (0,\infty)
\end{eqnarray}
on a two dimensional torus.  
Here
$\mathcal{L}$ is an anti-Hermitian operator
satisfying $\mathcal{L}^* = -\mathcal{L}$ where
$\mathcal{L}^*$ is the adjoint in $L^2(\mathbf{T}^2)$.  
Equation~(\ref{eq:generalmodel}) describes a prototypical linear wave
system and the advection equation with velocity $c$ is the simplest example:
\begin{eqnarray}
\label{eq:specificmodel}
\partial_t v+c\cdot \nabla v =0,
\quad \forall (x,t) \in \mathbf{T}^2 \times (0,\infty).
\end{eqnarray}
The state estimation problem for Equation~(\ref{eq:generalmodel}) requires to
find the `best' estimate of the solution $v(t)$ (shorthand for
$v(\cdot,t)$) given a random initial condition $v_0$ and a set of noisy
observations, called data.
Suppose the data is collected at discrete times $t_n=n\Delta t$,
then we assume that the entire $v_n = v(t_n)$ solution on the
torus is observed with additive noise $\eta_n$ at time $t_n$, and further
that the $\eta_n$ are independent for different $n$.
Realizations of the noise $\eta_n$ are $L^2(\mathbf{T}^2)$-valued random fields
and the observations $y_n$ are given by
\begin{eqnarray}
\label{eq:discreteobservation}
y_n = v_n+\eta_n = e^{-t_n \mathcal{L}}v_0+\eta_n, \quad \forall x\in \mathbf{T}^2.
\end{eqnarray}
Here $e^{-t\mathcal{L}}$ denotes the forward solution operator for (\ref{eq:generalmodel})
through $t$ time units.
Let $Y_N= \{ y_1,\ldots,y_N \}$ be the collection of data up to time $t_N$,
then we are interested in finding
the conditional distribution $\mathbb{P}\left( v_n \vert Y_N \right)$
on the Hilbert space $L^2(\mathbf{T}^2)$.
If $n=N$, this is called the filtering
problem, if $n<N$ it is the smoothing problem and for $n>N$, the
prediction problem. We here emphasize that all of the problems 
are equivalent for our deterministic system
in that any one measure defines the other simply by a push forward
under the linear map defined by Equation~(\ref{eq:generalmodel}).

In general calculation of the filtering distribution
$\mathbb{P}(v_n|Y_n)$ is computationally challenging
when the state space for $v_n$ is large, as it is here. 
A key idea is to estimate the signal $v_n$
through sequential updates consisting of a two-step process:
prediction by time evolution, and analysis through data assimilation.  
We first perform a one-step statistical prediction to obtain
$\mathbb{P}\left( v_{n+1} \vert Y_{n} \right)$
from $\mathbb{P}\left( v_{n} \vert Y_{n} \right)$ through some 
forward operator.  This is followed by an analysis step
which corrects the probability
distribution on the basis of the statistical input of noisy observations of the system
using Bayes rule:
\begin{eqnarray}
\label{eq:Bayesian}
\frac{\mathbb{P}(v_{n+1} \vert Y_{n+1})}{\mathbb{P}(v_{n+1} \vert Y_{n})} \propto
\mathbb{P}\left( y_{n+1} \vert v_{n+1} \right).
\end{eqnarray}
This relationship exploits the assumed independence of the observational
noises $\eta_n$ for different $n$.
In our case, where the signal $v_n$ is a function,
this identity should be interpreted
as providing the Radon-Nikodym derivative (density)
of the measure $\mathbb{P}(dv_{n+1}|Y_{n+1})$ with respect to $\mathbb{P}(dv_{n+1}|Y_{n})$
~\cite{cotter2009bayesian}.

In general implementation of this scheme is non-trivial 
as it requires approximation of the probability distributions
at each step indexed by $n$. In the case of infinite
dimensional dynamics this may be particularly challenging.
However, for linear dynamical systems such as
(\ref{eq:generalmodel}), together with
linear observations (\ref{eq:discreteobservation})
subject to Gaussian observation noise $\eta_n$
this may be achieved by use of the Kalman filter
\cite{Kalman60}.
Our work in this paper
begins with Theorem~\ref{inflinearfilter},
which is a straightforward
extension of the traditional Kalman filter theory in finite dimension
to measures on an infinite dimensional function space.
For reference, basic results concerning Gaussian measures 
required for this paper are gathered in \ref{app:A}.

\begin{thm}
\label{inflinearfilter}
Let $v_0$ be distributed
according to a Gaussian measure 
$\mu_0 = \mathcal{N}\left( m_0,\mathcal{C}_0 \right)$ 
on $L^2(\mathbf{T}^2)$
and let $\lbrace \eta_n \rbrace_{n \in \mathbb{N}}$ 
be i.i.d. draws from the
$L^2(\mathbf{T}^2)$-valued Gaussian measure $\mathcal{N}\left( 0, \Gamma \right)$.  
Assume further that $v_0$ and $\lbrace \eta_n \rbrace_{n \in \mathbb{N}}$ are independent of one
another, and that $\mathcal{C}_0$ and $\Gamma$ are strictly positive.
Then the conditional distribution
$\mathbb{P}(v_n \vert Y_n) \equiv \mu^n$ is a Gaussian
$\mathcal{N}\left( m^n, \mathcal{C}^n \right)$ with mean and
covariance satisfying the recurrence relations
\begin{subequations}
\label{eq:discretefilterx}
\begin{eqnarray}
m^{n+1}  = & e^{-\Delta t \mathcal{L}}m^{n} \nonumber \\
	& -e^{-\Delta t \mathcal{L}}\mathcal{C}^{n}e^{-\Delta t \mathcal{L}^*}
\left( \Gamma + e^{-\Delta t \mathcal{L}}\mathcal{C}^{n}e^{-\Delta t \mathcal{L}^*} \right)^{-1}
\left( e^{-\Delta t \mathcal{L}}m^{n} - y_{n+1} \right),
\label{eq:discretekalmanfilter-a} \\
\mathcal{C}^{n+1}  = & e^{-\Delta t \mathcal{L}}\mathcal{C}^{n}e^{-\Delta t \mathcal{L}^*} \nonumber \\
& -e^{-\Delta t \mathcal{L}}\mathcal{C}^{n}e^{-\Delta t \mathcal{L}^*} \left( \Gamma + e^{-\Delta t \mathcal{L}}\mathcal{C}^{n}e^{-\Delta t \mathcal{L}^*} \right)^{-1} e^{-\Delta t \mathcal{L}}\mathcal{C}^{n}e^{-\Delta t \mathcal{L}^*},
\label{eq:discretekalmanfilter-b}
\end{eqnarray}
\end{subequations}
where $m^0=m_0$, $\mathcal{C}^0 = \mathcal{C}_0$.  
\end{thm}

\begin{proof}
Let
$m_{n|N} = \mathbb{E}(v_{n} \vert Y_N)$ and
$\mathcal{C}_{n|N} = \mathbb{E}\left[ (v_{n} - m_{n|N})(v_{n} - m_{n|N})^* \right]$
denote the mean and covariance operator of $\mathbb{P}(v_{n} \vert Y_N)$
so that $m^n=m_{n|n}$, $\mathcal{C}^n=\mathcal{C}_{n|n}$.
Now the prediction step reads
\begin{eqnarray}
\label{eq:prediction}
\eqalign
m_{n+1|n} &= \mathbb{E}(e^{-\Delta t \mathcal{L}} v_{n} \vert Y_n)
= e^{-\Delta t \mathcal{L}}m_{n|n}, \nonumber\\
\mathcal{C}_{n+1|n} &=
\mathbb{E}\left[ e^{-\Delta t \mathcal{L}} (v_{n} - m_{n|n})(v_{n} - m_{n|n})^*
e^{-\Delta t \mathcal{L}^*} \right]\nonumber\\
& =e^{-\Delta t \mathcal{L}}\mathcal{C}_{n|n}e^{-\Delta t \mathcal{L}^*}.
\end{eqnarray}
To get the analysis step,
choose $x_1=v_{n+1}\vert Y_n$ and $x_2=y_{n+1} \vert Y_n$
then $(x_1,x_2)$ is jointly Gaussian with mean $(m_{n+1|n},m_{n+1|n})$
and each components of the covariance operator for $(x_1,x_2)$ are given by
\begin{eqnarray*}
C_{11} &= \mathbb{E}\left[ (x_1-m_{n+1|n})(x_1-m_{n+1|n})^* \right]= \mathcal{C}_{n+1|n}, \\
C_{22} &= \mathbb{E}\left[ (x_2-m_{n+1|n})(x_2-m_{n+1|n})^* \right]= \Gamma+
\mathcal{C}_{n+1|n}, \\
C_{12} &= \mathbb{E}\left[ (x_1-m_{n+1|n})(x_2-m_{n+1|n})^* \right]= \mathcal{C}_{n+1|n} =C_{21}.
\end{eqnarray*}
Using Lemma~\ref{condgaussian} we obtain
\begin{eqnarray}
\label{eq:analysis}
m_{n+1|n+1} &= m_{n+1|n} - \mathcal{C}_{n+1|n}
\left( \Gamma + \mathcal{C}_{n+1|n} \right)^{-1} \left( m_{n+1|n} - y_{n+1} \right), \nonumber \\
\mathcal{C}_{n+1|n+1} &= \mathcal{C}_{n+1|n} - \mathcal{C}_{n+1|n}
\left( \Gamma + \mathcal{C}_{n+1|n} \right)^{-1} \mathcal{C}_{n+1|n}.
\end{eqnarray}
Combining Equations~(\ref{eq:prediction}) and (\ref{eq:analysis})
yields Equation~(\ref{eq:discretefilterx}).  
\end{proof}

Note that $\mu_n$, the distribution of $v_0|Y_n$, 
is also Gaussian and we denote its mean and covariance by 
$m_n$ and $\mathcal{C}_n$ respectively. 
The measure $\mu_n$ is the image of $\mu^n$ under the 
linear transformation $e^{t_n \mathcal{L}}$.
Thus we have $m_n = e^{t_n \mathcal{L}}m^n$ and
$\mathcal{C}_n=e^{t_n\mathcal{L}}\mathcal{C}^n e^{t_n\mathcal{L}^*}$.

In this paper we study a wave propagation
problem for which $\mathcal{L}$ is anti-Hermitian. Furthermore
we assume that both $\mathcal{C}_0$ and $\Gamma$ commute
with $\mathcal{L}$. Then the formulae (\ref{eq:discretefilterx})
simplify to give
\begin{subequations}
\label{eq:discretefilter}
\begin{eqnarray}
m^{n+1} & =  e^{-\Delta t \mathcal{L}}m^{n}-\mathcal{C}^{n}
\left( \Gamma + \mathcal{C}^{n} \right)^{-1}
\left( e^{-\Delta t \mathcal{L}}m^{n} - y_{n+1} \right),
\label{eq:discretekalmanfilter-a} \\
\mathcal{C}^{n+1} & =  \mathcal{C}^{n}
-\mathcal{C}^{n} \left( \Gamma + \mathcal{C}^{n} \right)^{-1} \mathcal{C}^{n}.
\label{eq:discretekalmanfilter-b}
\end{eqnarray}
\end{subequations}
The following gives explicit expressions for
$\left(m_n,\mathcal{C}_n\right)$ and $\left(m^n,\mathcal{C}^n\right)$,
based on the Equations (\ref{eq:discretefilter}).

\begin{corollary}
\label{generaltermsmoother}
Suppose that $\mathcal{C}_0$ and $\Gamma$ commute with the anti-Hermitian operator $\mathcal{L}$,
then the means and covariance operators of $\mu_n$ and $\mu^n$
are given by
\begin{subequations}
\label{eq:discretekalman}
\begin{eqnarray}
m_{n}
& = \left( nI+ \Gamma\mathcal{C}_{0}^{-1} \right)^{-1}
\left[ \Gamma\mathcal{C}_{0}^{-1} m_{0}
+ \sum_{l=0}^{n-1} e^{t_{l+1} \mathcal{L}}y_{l+1} \right],
\label{eq:discretekalman-a} \\
\mathcal{C}_n
& = \left( n \Gamma^{-1}+\mathcal{C}_{0}^{-1} \right)^{-1},
\label{eq:discretekalman-c} 
\end{eqnarray}
\end{subequations}
and $m^{n} = e^{-t_n \mathcal{L}}m_n$, $\mathcal{C}^n = \mathcal{C}_n$.
\end{corollary}

\begin{proof}
Assume for induction that $\mathcal{C}^n$ is invertible.
Then the identity
\begin{eqnarray*}
& \left( \Gamma^{-1}+ \left(\mathcal{C}^n\right)^{-1} \right)
\left(\mathcal{C}^n-\mathcal{C}^n\left(\Gamma+\mathcal{C}^n\right)^{-1}\mathcal{C}^n\right) \\
&=
\left( \Gamma^{-1}+ \left(\mathcal{C}^n\right)^{-1} \right)
\left(\mathcal{C}^n-\mathcal{C}^n
\left[
\left(\mathcal{C}^n\right)^{-1} \left(\Gamma^{-1}+\left(\mathcal{C}^n\right)^{-1} \right)^{-1}
\Gamma^{-1} \right]
\mathcal{C}^n\right)
=I
\end{eqnarray*}
leads to $(\mathcal{C}^{n+1})^{-1} =  \Gamma^{-1}+ (\mathcal{C}^n)^{-1}$
from Equation~(\ref{eq:discretekalmanfilter-b}), and hence $\mathcal{C}^{n+1}$ is invertible.
Then Equation~(\ref{eq:discretekalman-c}) follows by induction.
By applying $e^{t_{n+1}\mathcal{L}}$ to Equation~(\ref{eq:discretekalmanfilter-a})
we have
\begin{eqnarray*}
m_{n+1} = m_{n}-\mathcal{C}^{n}\left( \Gamma+\mathcal{C}^{n} \right)^{-1}
\left( m_{n}
-e^{t_{n+1} \mathcal{L}}y_{n+1}
\right).
\end{eqnarray*}
After using
$\mathcal{C}^{n}(\Gamma+\mathcal{C}^{n})^{-1}
=\left( (n+1)I+\Gamma \mathcal{C}_{0}^{-1}\right)^{-1}$
from Equation~(\ref{eq:discretekalman-c}),
we obtain
the telescoping series
\begin{eqnarray*}
\left( \left(n+1\right)I+\Gamma \mathcal{C}_0^{-1} \right) m_{n+1}
=\left( nI+\Gamma \mathcal{C}_0^{-1} \right) m_{n} + e^{t_{n+1}\mathcal{L}}y_{n+1}
\end{eqnarray*}
and 
Equation~(\ref{eq:discretekalman-a}) follows.  
The final observations follow since
$m_n = e^{t_n \mathcal{L}}m^n$ and
$\mathcal{C}_n=e^{t_n\mathcal{L}}\mathcal{C}^n e^{t_n\mathcal{L}^*}=\mathcal{C}^n$.
\end{proof}

\subsection{Equivalence of measures}
Now suppose we have a set of specific data and they are a single realization of
the observation process~(\ref{eq:discreteobservation}),
\begin{eqnarray}
\label{eq:singlediscreteobservation}
y_n(\omega)=e^{-t_n\mathcal{L}}u+\eta_n(\omega).
\end{eqnarray}
Here $\omega$ is an element of the probability space $\Omega$ 
generating the entire noise signal
$\lbrace \eta_n \rbrace_{n \in \mathbb{N}}.$ 
We assume that the initial condition $u \in L^2(\mathbf{T}^2)$ for the true signal
$e^{-t_n\mathcal{L}}u$
is non-random and hence independent of $\mu_0$.
We insert the fixed (non-random) instances of the data~(\ref{eq:singlediscreteobservation})
into the formulae for $\mu_0$, $\mu_n$ and $\mu^n$,
and will prove all three measures are equivalent.
Recall that two measures are said to be equivalent
if they are mutually absolutely continuous,
and singular if they are concentrated on disjoint sets~\cite{Prato92}.  
Lemma~\ref{FH} (Feldman-Hajek) tells us
the conditions under which two Gaussian measures are equivalent.

Before stating the theorem, we need to introduce some notation
and assumptions.  
Let $\phi_k(x) = e^{2\pi ik \cdot x}$,
where $k = (k_1,k_2) \in \mathbb{Z}\times\mathbb{Z} \equiv \mathbb{K}$,
be a standard orthonormal basis for $L^2(\mathbf{T}^2)$
with respect to the standard inner product
$(f,g) = \int_{\mathbf{T}^2} f\bar{g}\,dxdy$
where the upper bar denotes the complex conjugate.

\begin{assumptions}
\label{asp:1}
The operators $\mathcal{L}, \mathcal{C}_0$ and $\Gamma$ 
are diagonalizable in the basis defined by the $\phi_k.$ 
The two covariance operators, $\mathcal{C}_0$ and $\Gamma$,
have positive eigenvalues $\lambda_k>0$ and $\gamma_k>0$ 
respectively: 
\begin{eqnarray}
\label{eq:eigen}
\mathcal{C}_0\phi_k & = \lambda_k \phi_k, \nonumber\\
\Gamma\phi_k & = \gamma_k \phi_k.
\end{eqnarray}
The eigenvalues $\ell_k$ of $\mathcal{L}$ have zero real
parts and $\mathcal{L}\phi_0=0.$
\end{assumptions}

This assumption on the simultaneous diagonalization of the operators implies
the commutativity of $\mathcal{C}_0$ and $\Gamma$ with $\mathcal{L}$.
Therefore we have Corollary~\ref{generaltermsmoother} and
Equations~(\ref{eq:discretekalman}) can be used to study the large $n$ behaviour of the $\mu_n$ and $\mu^n$.
We believe that it might be possible to 
obtain the subsequent results without Assumption~\ref{asp:1},
but to do so would require significantly different
techniques of analysis; the simultaneously diagonalizable
case allows a straightforward analysis in which the
mechanisms giving rise to the results are easily understood.

Note that, since $\mathcal{C}_0$ and $\Gamma$ are
the covariance operators of Gaussian measures on
$L^2(\mathbf{T}^2),$ it follows from Lemma~\ref{gaussiancovop}
that the $\lambda_k$, $\gamma_k$ are summable:
i.e., $\sum_{k\in \mathbb{K}}\lambda_k < \infty$,
$\sum_{k\in \mathbb{K}}\gamma_k < \infty$.
We define $H^s(\mathbf{T}^2)$ to be the Sobolev space
of periodic functions with $s$ weak derivatives and 
$$\left\| \cdot \right\|_{H^s(\mathbf{T}^2)} \equiv \sum_{k \in \mathbb{K}^+}
|k|^{2s}\lvert (\cdot, \phi_k) \rvert^2+ \lvert (\cdot, \phi_0) \rvert^2$$
where
$\mathbb{K}^+ = \mathbb{K} \backslash \lbrace (0,0) \rbrace$
noting that this norm reduces to the
usual $L^2$ norm when $s=0$. We denote by 
$\left\| \cdot \right\|$ the standard Euclidean norm.

\begin{thm}
\label{FH2}
If $\sum_{k \in \mathbb{K}} \lambda_k / \gamma_k^2 < \infty$,
then
the Gaussian measures $\mu_0$, $\mu_n$ and $\mu^n$ on the Hilbert space
$L^2(\mathbf{T}^2)$ are equivalent $\eta-$a.s.  
\end{thm}
\begin{proof}
We first show the equivalence between $\mu_0$ and $\mu_n$.
For $h = \sum_{k \in \mathbb{K}} h_k \phi_k$
we get, from Equations~(\ref{eq:discretekalman-c}) and (\ref{eq:eigen}),
\begin{eqnarray*}
\frac{1}{c^+}
\leq \frac{\left( h, \mathcal{C}_n h\right)}{\left( h, \mathcal{C}_0 h\right)}
= \frac{\sum_{k \in \mathbb{K}}
\left( n\gamma_k^{-1}+\lambda_k^{-1}\right)^{-1}h_k^2}
{\sum_{k\in \mathbb{K}} \lambda_k h_k^2}
\leq 1,
\end{eqnarray*}
where $c^+=\sup_{k \in \mathbb{K}} \left( n\lambda_k/\gamma_k+1 \right)$.
We have $c^+ \in [1, \infty)$
because $\sum_{k \in \mathbb{K}} \lambda_k/\gamma_k^2 < \infty$
and $\Gamma$ is trace-class.
Then the first condition for Feldman-Hajek is satisfied by Lemma~\ref{FH1}.

For the second condition,
take
$\{ \mathbf{g}_k^{l+1} \}_{k \in \mathbb{K}}$
where
$l=0,\ldots,n-1$,
to be a sequence of
complex-valued
unit Gaussians
independent except for the condition 
$\mathbf{g}_{-k}^l=\bar{\mathbf{g}}_{k}^l$.
This constraint ensures that 
the Karhunen-Lo\`{e}ve expansion 
\begin{eqnarray*}
e^{t_{l+1} \mathcal{L}}\eta_{l+1}
= \sum_{k \in \mathbb{K}} \sqrt{\gamma_k}\mathbf{g}_k^{l+1} \phi_k,
\end{eqnarray*}
is distributed according to the
real-valued Gaussian measure $\mathcal{N}\left(0,\Gamma\right)$,
and is independent for different values of $l$. 
Thus
\begin{eqnarray*}
\fl \left\| m_n - m_0 \right\|^2_{\mathcal{C}_0}
& \equiv \left\| \mathcal{C}_0^{-\frac{1}{2}} (m_n-m_0) \right\|^2_{L^2(\mathbf{T}^2)}\\
& = \sum_{k \in \mathbb{K}}
\left( \frac{\lambda_k^{-{1}/{2}}}{n+\gamma_k/\lambda_k} \right )^2
\left\lvert
-n (m_0,\phi_k)+ n (u,\phi_k) 
 +\sqrt{\gamma_k}
\sum_{l=0}^{n-1}\mathbf{g}_k^{l+1}
\right\rvert^2 \\
& \leq C(n)\sum_{k \in \mathbb{K}} \frac{\lambda_k}{\gamma_k^2} < \infty,
\end{eqnarray*}
where
$$C(n) \equiv \sup_{k\in \mathbb{K}} \left\lvert -n (m_0,\phi_k)+ n (u,\phi_k)
 +\sqrt{\gamma_k}
\sum_{l=0}^{n-1}\mathbf{g}_k^{l+1} \right\rvert ^2<\infty$$
$\eta-$a.s.
from the strong law of large numbers~\cite{Varadhan}.

For the third condition, we need to show
the set of eigenvalues of $T$, where
\begin{eqnarray*}
T\phi_k = \left( \mathcal{C}_n^{-\frac{1}{2}}\mathcal{C}_0
\mathcal{C}_n^{-\frac{1}{2}}-I \right) \phi_k
=n \left( \frac{\lambda_k}{\gamma_k} \right) \phi_k,
\end{eqnarray*}
are square-summable.
This is satisfied since $\mathcal{C}_0$ is trace-class:
\begin{eqnarray*}
\sum_{k \in \mathbb{K}} \frac{\lambda_k^2}{\gamma_k^2}
\leq \left( \sup_{k \in \mathbb{K}} \lambda_k\right)
\sum_{k \in \mathbb{K}} \frac{\lambda_k}{\gamma_k^2} < \infty.
\end{eqnarray*}

The equivalence between $\mu_0$ and $\mu^n$
is immediate because $m^n$ is the image of $m_n$ under a unitary map
$e^{-t_n\mathcal{L}}$, and $\mathcal{C}^n = \mathcal{C}_n$.
\end{proof}

To illustrate the conditions of the theorem,
let $(-\triangle)$ denote the negative Laplacian with domain
$\mathcal{D}(-\triangle)= H^2(\mathbf{T}^2)$.
Assume that 
$\mathcal{C}_0 \propto \left(-\triangle +k_A I \right)^{-A}$ and
$\Gamma \propto \left(-\triangle +k_B I \right)^{-B}$,
where the conditions $A>1$, $B>1$
and $k_A, k_B > 0$ ensure, respectively, that the
two operators are trace-class and positive-definite.
Then the condition
$\sum_{k \in \mathbb{K}} \lambda_k / \gamma_k^2 < \infty$
reduces to
$2{B}\leq {A}-1.$

\section{Measure Consistency and Time-Independent Model Error}
\label{MCTI}
In this section we study the
large data
limit of the smoother $\mu_n$ and the filter
$\mu^n$ in Corollary~\ref{generaltermsmoother}
for large $n$.
We study measure consistency, namely the question of whether the filtering or smoothing
distribution converges to a Dirac measure on the true signal as $n$ increases. We study situations
where the data is generated as a single realization of the statistical model itself, so there is no model
error, and situations where model error is present. 

\subsection{Data model}
Suppose that
the true signal, denoted $v_n'=v'(\cdot,t_n)$, can be different from the solution $v_n$
computed via the model~(\ref{eq:generalmodel}) and instead solves
\begin{eqnarray}
\label{eq:generaltruthmodel}
\partial_t v'+\mathcal{L}' v' =0,
\quad
v'(0)=u
\end{eqnarray}
for another anti-Hermitian operator $\mathcal{L}'$ on $L^2(\mathbf{T}^2)$
and fixed $u \in L^2(\mathbf{T}^2)$.
We further assume possible error in the observational noise model 
so that the actual noise in the data
is not $\eta_n$ but $\eta_n'$.
Then, instead of $y_n$ given by (\ref{eq:singlediscreteobservation}),
what we actually incorporate into the filter is the true data,
a single realization $y_n'$ determined by $v_n'$ and $\eta_n'$ as follows:
\begin{eqnarray}
\label{eq:discretetruthobservation}
y'_n = v'_n+\eta_n' = e^{-t_n \mathcal{L}'}u+\eta_n'.
\end{eqnarray}
We again use $e^{-t\mathcal{L}'}$ to denote the forward solution operator, now for
(\ref{eq:generaltruthmodel}),
through $t$ time units.
Note that each realization (\ref{eq:discretetruthobservation})
is an element 
in the probability space $\Omega'$, which is independent of $\Omega$.
For the data $Y_n' = \{ y_1',\ldots,y_n' \}$, let $\mu_n'$ be the measure
$\mathbb{P}\left(v_0|Y_n=Y_n'\right)$. This measure is
Gaussian and is determined by the mean in Equation~(\ref{eq:discretekalman-a}) with $y_l$
replaced by $y_l'$, which we relabel as $m_n'$,
and the covariance operator in Equation~(\ref{eq:discretekalman-c})
which does not depend on the data so we retain
the notation $\mathcal{C}_n$.
Clearly, using (\ref{eq:discretetruthobservation}) we obtain
\begin{eqnarray}
\label{eq:posmean}
m_{n}'
= \left( nI+ \Gamma\mathcal{C}_{0}^{-1} \right)^{-1}
\left[ \Gamma\mathcal{C}_{0}^{-1} m_{0}
+ \sum_{l=0}^{n-1}\left( e^{t_{l+1} \left(\mathcal{L}-\mathcal{L}' \right)} u
+ e^{ t_{l+1} \mathcal{L}}\eta_{l+1}'
\right) \right],
\end{eqnarray}
where $m_0'=m_0$.

This conditioned mean
$m_n'$ differs from $m_n$
in that $e^{t_{l+1} \left(\mathcal{L}-\mathcal{L}' \right)} u$ and $\eta_{l+1}'$
appear instead of $u$ and $\eta_{l+1}$, respectively.
The reader will readily generalize both the
statement and proof of Theorem~\ref{FH2} to show
the equivalence of $\mu_0$, $\mu_n'$ and
$(\mu')^n \equiv \mathbb{P}(v_n|Y_n=Y_n')=\mathcal{N}\left( (m')^n, \mathcal{C}^n \right),$
showing
the well-definedness of these conditional distributions even
with errors in both forward model and observational noise.
We now study the large data limit for the filtering problem, in the idealized
scenario where $\mathcal{L}=\mathcal{L}'$ (Theorem~\ref{limitthm1})
and the more realistic scenario with model error so that $\mathcal{L}
\neq \mathcal{L}'$ (Theorem~\ref{limitthm2}). We allow possible 
model error in the observation noise for both theorems, 
so that the noises $\eta_n'$ are
draws from i.i.d. Gaussians $\eta_n' \sim
\mathcal{N}(0,\Gamma')$ and $\Gamma'\phi_k=\gamma_k' \phi_k$.
Even though $\gamma_k'$ (equivalently $\Gamma'$) and $\mathcal{L}'$ are not exactly known in most practical
situations, their asymptotics can be predicted
within a certain degree of accuracy.
Therefore, without limiting the applicability of the theory,
the following are assumed in all subsequent theorems and corollaries in this paper:

\begin{assumptions}
\label{ass1}
There are positive real numbers $s, \kappa \in \mathbb{R}^+$ such that:
\begin{enumerate}
\item $\sum_{k \in \mathbb{K}} |k|^{2s}\gamma_k < \infty$,
$\sum_{k \in \mathbb{K}} |k|^{2s}\gamma_k' < \infty$;
\item ${\gamma_k}/{\lambda_k} = O\left(|k|^\kappa \right)$;
\item $m_0, u \in H^{s+\kappa}\left(\mathbf{T}^2 \right)$.
\end{enumerate}
\end{assumptions}
Then Assumptions~\ref{ass1}(1)
imply that $\eta \sim \mathcal{N}(0,\Gamma)$ and $\eta' \sim \mathcal{N}(0,\Gamma')$
are in $H^s(\mathbf{T}^2)$ since
${\mathbb E}\|\eta\|_{H^s(\mathbf{T}^2)}^2<\infty$ and
${\mathbb E}\|\eta'\|_{H^s(\mathbf{T}^2)}^2<\infty.$

\subsection{Limit of the conditioned measure without model error}
We first study the large data limit of the measure $\mu_n'$ without model error.

\begin{thm}
\label{limitthm1}
For the statistical model~(\ref{eq:generalmodel}) and (\ref{eq:discreteobservation}),
suppose that the data $Y_n'=\lbrace y_1',\cdots,y_n' \rbrace$ is created from
(\ref{eq:discretetruthobservation}) with $\mathcal{L}=\mathcal{L}'$.
Then, as $n \to \infty$,
$\mathbb{E}(v_0|Y_n')=m_n' \to u$ in the sense that
\begin{subequations}
\label{eq:discov}
\begin{eqnarray}
&\left\| m_n'-u \right\|_{L^2(\Omega'; H^s(\mathbf{T}^2))} = O\left(n^{-\frac{1}{2}}\right),
\label{eq:discov-a} \\
&\left\| m_n'-u \right\|_{H^s(\mathbf{T}^2)} = o\left(n^{-\theta}\right)
\quad \Omega'-a.s.,
\label{eq:discov-b}
\end{eqnarray}
\end{subequations}
for the probability space $\Omega'$ generating the true observation noise $\lbrace \eta_n'
\rbrace_{n \in \mathbb{N}}$,
and for any non-negative $\theta<1/2$.
Furthermore, $\mathcal{C}_n \to 0$ in the sense that its operator norm from
$L^2(\mathbf{T}^2)$ to $H^s(\mathbf{T}^2)$ satisfies 
\begin{eqnarray}
\label{eq:discovop}
\left\| \mathcal{C}_n \right\|_{\mathbf{L}(L^2(\mathbf{T}^2);H^s(\mathbf{T}^2))}=O(n^{-1}).
\end{eqnarray}
\end{thm}

\begin{proof}
From Equation~(\ref{eq:posmean}) with $\mathcal{L}'= \mathcal{L}$, we have
\begin{eqnarray*}
\left( nI+\Gamma \mathcal{C}_{0}^{-1}\right) m_{n}'
= \Gamma \mathcal{C}_{0}^{-1} m_{0} + nu
+\sum_{l=0}^{n-1} e^{t_{l+1}\mathcal{L} } \eta_{l+1}',
\end{eqnarray*}
thus
\begin{eqnarray*}
\left( nI+\Gamma \mathcal{C}_{0}^{-1}\right) (m_{n}'-u)
= \Gamma \mathcal{C}_{0}^{-1} (m_{0}-u )
+\sum_{l=0}^{n-1} e^{t_{l+1} \mathcal{L} }\eta_{l+1}'.
\end{eqnarray*}
Take $\{ (\mathbf{g}_k')^{l+1} \}_{k \in \mathbb{K}}$ where
$l=0,\ldots,n-1$, to be an i.i.d. sequence of
complex-valued
unit Gaussians
subject to the constraint that
$(\mathbf{g}_{-k}')^{l}=(\bar{\mathbf{g}}_{k}')^{l}$.
Then the Karhunen-Lo\`{e}ve expansion for
$e^{t_{l+1}\mathcal{L}}\eta_{l+1}' \sim \mathcal{N}\left(0,\Gamma' \right)$
is given by
\begin{eqnarray}
\label{eq:newKLexpn}
e^{t_{l+1} \mathcal{L}}\eta'_{l+1}
= \sum_{k \in \mathbb{K}} \sqrt{\gamma_k'}(\mathbf{g}_k')^{l+1} \phi_k.
\end{eqnarray}
It follows that
\begin{eqnarray*}
\fl \left\| m_n'-u \right\|^2_{L^2(\Omega'; H^s(\mathbf{T}^2))}
& = \mathbb{E} \left\| m_n'-u \right\|^2_{H^s(\mathbf{T}^2)}  \\
& = \sum_{k \in \mathbb{K}^+} |k|^{2s}
\mathbb{E} \lvert (m_n'-u,\phi_k) \rvert^2+ \mathbb{E} \lvert (m_n'-u,\phi_0)\rvert^2
\\
& = \sum_{k \in \mathbb{K}^+} \frac{|k|^{2s} }{\left( n+\gamma_k/\lambda_k\right)^2}
\left[ \left( \frac{\gamma_k}{\lambda_k} \right)^2 \lvert (m_0-u,\phi_k) \rvert^2 +\gamma_k'
n\right] \\
& \quad + \frac{1}{\left( n+\gamma_0/\lambda_0\right)^2}
\left[ \left( \frac{\gamma_0}{\lambda_0} \right)^2 \lvert (m_0-u,\phi_0) \rvert^2 +\gamma_0'
n\right] \\
& \leq \sum_{k \in \mathbb{K}^+} |k|^{2s} n^{-2}
\left[ \left( \frac{\gamma_k}{\lambda_k} \right)^2 \lvert (m_0-u,\phi_k) \rvert^2 +\gamma_k'
n\right] \\
& \quad + n^{-2}
\left[ \left( \frac{\gamma_0}{\lambda_0} \right)^2 \lvert (m_0-u,\phi_0) \rvert^2 +\gamma_0'
n\right] \\
& \leq \left( C\left\| m_0 -u \right\|_{H^{s+\kappa}(\mathbf{T}^2)}\right) n^{-2}
+\left( \sum_{k \in \mathbb{K}} |k|^{2s}\gamma_k' \right)n^{-1},
\end{eqnarray*}
and so we have Equation~(\ref{eq:discov-a}).
Here and throughout the paper, $C$ is a constant that may change from line to line.

Equation~(\ref{eq:discov-b}) follows from the Borel-Cantelli Lemma
\cite{Varadhan}:
for arbitrary $\epsilon > 0$, we have
\begin{eqnarray*}
&\sum_{n \in \mathbb{N}}
\mathbb{P}\left( n^\theta \left\| m_n'-u \right\|_{H^s(\mathbf{T}^2)} > \epsilon \right)
= \sum_{n \in \mathbb{N}}
\mathbb{P}\left( n^{2r\theta} \left\| m_n'-u \right\|_{H^s(\mathbf{T}^2)}^{2r}
> \epsilon^{2r} \right) \\
&\qquad \qquad \leq
\sum_{n \in \mathbb{N}}
\frac{n^{2r\theta}}{\epsilon^{2r}}
\mathbb{E}  \left\| m_n'-u \right\|^{2r}_{H^s(\mathbf{T}^2)} 
\qquad \mbox{(Markov inequality)} \\
&\qquad \qquad \leq 
\sum_{n \in \mathbb{N}}
C n^{2r\theta}
\left( \mathbb{E} \left\| m_n'-u \right\|^{2}_{H^s(\mathbf{T}^2)} \right)^r
\qquad \mbox{(Lemma~\ref{gaussianmoment})} \\
&\qquad \qquad \leq
\sum_{n \in \mathbb{N}}
\frac{C}{n^{r(1-2\theta)}} < \infty
\qquad \mbox{(by (\ref{eq:discov-a}))}
\end{eqnarray*}
and if $\theta \in (0,1/2)$ then we can choose $r$ such that
$r(1-2\theta)>1$.

Finally, for $h=\sum_{k\in \mathbb{K}} h_k \phi_k$,
\begin{eqnarray*}
\left\| \mathcal{C}_n \right\|^2_{\mathbf{L}(L^2(\mathbf{T}^2);H^s(\mathbf{T}^2))}
& = \sup_{\left\| h \right\|_{L^2(\mathbf{T}^2)} \leq 1}
\left\| \mathcal{C}_n h  \right\|^2_{H^s(\mathbf{T}^2)} \\
& = \sup_{\left\| h \right\|_{L^2(\mathbf{T}^2)} \leq 1}
\sum_{k\in \mathbb{K}} |k|^{2s} \left\lvert n\gamma_k^{-1}+\lambda_k^{-1} \right\rvert^{-2}|h_k|^2\\
& \leq C \sum_{k\in \mathbb{K}} |k|^{2s} \left\lvert n\gamma_k^{-1}+\lambda_k^{-1} \right\rvert^{-2}\\
& \leq \frac{C}{n^2} \sum_{k\in \mathbb{K}} |k|^{2s}\gamma_k^2.
\end{eqnarray*}
and use the fact
that $\sup_{k \in \mathbb{K}} \gamma_k<\infty,$ since $\Gamma$
is trace-class, together with Assumptions~\ref{ass1}(1)
to get the desired convergence rate.
\end{proof}

\begin{corollary}
\label{almosteverwhere}
Under the same assumptions as Theorem~\ref{limitthm1},
if Equation~(\ref{eq:discov-b}) holds with $s>1$, then
\begin{eqnarray}
\label{eq:discov-c}
\left\| m_n'-u \right\|_{L^\infty(\mathbf{T}^2)} = o\left(n^{-\theta}\right)
\quad \Omega'-a.s.,
\end{eqnarray}
for any non-negative $\theta<1/2$.
\end{corollary}
\begin{proof}
Equation~(\ref{eq:discov-c}) is immediate from Equation~(\ref{eq:discov-b}) and 
the Sobolev embedding,
\begin{eqnarray*}
\left\| \cdot \right\|_{L^\infty(\mathbf{T}^2)}
\leq C \left\| \cdot \right\|_{H^s(\mathbf{T}^2)},
\end{eqnarray*}
when $s > d/2 =1$ since $d=2$ is the dimension of the domain.
\end{proof}

\begin{corollary}
\label{limitwithoutmodelerror}
Under the same assumptions as Theorem~\ref{limitthm1},
as $n\to \infty$,
the $\Omega'-a.s.$ weak convergence $\mu_n' \Rightarrow \delta_u$ holds
in $L^2(\mathbf{T}^2)$.
\end{corollary}

\begin{proof}
To prove this result we apply Example 3.8.15
in \cite{bogachev1998gaussian}. This shows that for
weak convergence of $\mu_n' = \mathcal{N}\left(m_n',\mathcal{C}_n \right)$
a Gaussian measure on $\mathcal{H}$
to a limit measure $\mu = \mathcal{N}\left(m,\mathcal{C}\right)$ on $\mathcal{H}$,
it suffices to show that
$m_n' \to m$ in $\mathcal{H}$, that $\mathcal{C}_n \to
\mathcal{C}$ in $\mathbf{L}(\mathcal{H},\mathcal{H})$
and that second moments converge. Note, also, that
Dirac measures, and more generally semi-definite
covariance operators, are included in the definition of
Gaussian. The convergence of the means and the
covariance operators follows from Equations~(\ref{eq:discov-b})
and (\ref{eq:discovop})
with $m=u$, $\mathcal{C}=0$
and $\mathcal{H}=L^2(\mathbf{T}^2)$.
The convergence of the second comments follows if the
trace of $\mathcal{C}_n$ converges to zero. From
(\ref{eq:discretekalman-c}) it follows that the trace
of $\mathcal{C}_n$ is bounded by $n^{-1}$ multiplied
by the trace of $\Gamma.$ But $\Gamma$ is trace-class
as it is a covariance operator on $L^2(\mathbf{T}^2)$
and so the desired result follows.
\end{proof}

In fact,
the weak convergence in the Prokhorov metric between $\mu_n'$ and $\delta_u$ holds,
and
the methodology in~\cite{NP08} could be used to quantify
its rate of convergence.

We now obtain the large data limit of the filtering distribution $(\mu')^n$ without model error
from the smoothing limit~(\ref{eq:discov}).
Recall that this measure is the Gaussian ${\mathcal N}\bigl(
(m')^n,{\mathcal C}_n\bigr).$

\begin{thm}
\label{NEW1}
Under the same assumptions as Theorem~\ref{limitthm1},
as $n\to \infty$,
$(m')^n-v'_n \to 0$ in the sense that
\begin{subequations}
\label{eq:disfilternomodelerror}
\begin{eqnarray}
&\left\| (m')^n-v'_n \right\|_{L^2(\Omega'; H^s(\mathbf{T}^2))} = O\left(n^{-\frac{1}{2}}\right),
\label{eq:disfilternomodelerror-a} \\
&\left\| (m')^n-v'_n \right\|_{H^s(\mathbf{T}^2)} = o\left(n^{-\theta}\right)
\quad \Omega'-a.s.,
\label{eq:disfilternomodelerror-b}
\end{eqnarray}
\end{subequations}
for any non-negative $\theta<1/2$.
\end{thm}
\begin{proof}
Equation~(\ref{eq:disfilternomodelerror-b}) follows from
\begin{eqnarray*}
\left\| (m')^n -v'_n \right\|_{H^s(\mathbf{T}^2)}
&= \left\| e^{-t_n\mathcal{L}}m_n' - e^{-t_n\mathcal{L}'}u
\right\|_{H^s(\mathbf{T}^2)} \\
&= \left\| e^{-t_n\mathcal{L}}\left(m_n' -u\right)
\right\|_{H^s(\mathbf{T}^2)} \\
& \leq \left\| e^{-t_n\mathcal{L}} \right\|_{\mathbf{L}(H^s(\mathbf{T}^2);H^s(\mathbf{T}^2))}
\left\| m_n' - u \right\|_{H^s(\mathbf{T}^2)}\\
& =\left\| m_n' - u \right\|_{H^s(\mathbf{T}^2)}.
\end{eqnarray*}
Then Equation~(\ref{eq:disfilternomodelerror-a}) follows from
\begin{eqnarray*}
\left\| (m')^n-v'_n \right\|^2_{L^2(\Omega'; H^s(\mathbf{T}^2))}
& = \mathbb{E} \left\| (m')^n -v'_n \right\|_{H^s(\mathbf{T}^2)}^2 \\
& \leq
\mathbb{E} \left\| m_n' - u \right\|_{H^s(\mathbf{T}^2)}^2.
\end{eqnarray*}
\end{proof}

The following corollary has 
the same as for Corollary~\ref{limitwithoutmodelerror},
and so we omit it.

\begin{corollary}
\label{limitwithoutmodelerror2}
Under the same assumptions as Theorem~\ref{limitthm1},
as $n\to \infty$,
the $\Omega'-a.s.$ weak convergence $(\mu')^n - \delta_{v_n'} \Rightarrow 0$ holds
in $L^2(\mathbf{T}^2)$.
\end{corollary}

Theorem~\ref{NEW1} and Corollary~\ref{limitwithoutmodelerror2}
are consistent with known results concerning
large data limits of the finite dimensional Kalman filter 
shown in \cite{kumar1986stochastic}.
However, Equations~(\ref{eq:disfilternomodelerror}) and (\ref{eq:discovop}) provide
convergence rates in an infinite dimensional
space and hence cannot be derived from the
finite dimensional theory; mathematically this is because
the rate of convergence in each Fourier mode 
will depend on the wavenumber $k$, 
and the infinite dimensional analysis
requires this dependence to be tracked and quantified,
as we do here.

\subsection{Limit of the conditioned measure with time-independent model error}
The previous subsection shows measure consistency results for data generated by the same PDE as
that used in the filtering model. In this section, we study the consequences of using data generated
by a different PDE.
It is important to point out that,
in view of Equation~(\ref{eq:posmean}),
the limit of $m_n'$ is determined by the time average of
$e^{t_{l+1}\left(\mathcal{L}-\mathcal{L}' \right)}u$, i.e., 
\begin{eqnarray}
\label{eq:averageadvection}
\frac{1}{n} \sum_{l=0}^{n-1} e^{t_{l+1}\left(\mathcal{L}-\mathcal{L}' \right)}u.
\end{eqnarray}
For general anti-Hermitian $\mathcal{L}$ and $\mathcal{L}'$,
obtaining an analytic expression
for the limit of
the average~(\ref{eq:averageadvection}),
as $n \to \infty$, is very hard.
Therefore in the remainder of
the section we examine
the case in which
$\mathcal{L}=c\cdot \nabla$ and $\mathcal{L}'=c'\cdot \nabla$ with different
constant wave velocities $c$ and $c'$, respectively.
A highly nontrivial filter divergence takes place even in this 
simple example.

We use the notation 
$\mathcal{F}_{(p,q)} f \equiv
\sum_{\left({k_1}/{p},{k_2}/{q}\right)
\in \mathbb{Z}\times \mathbb{Z}}
(f,\phi_k) \phi_k$
for part of the Fourier series of $f \in L^2\left( \mathbf{T}^2\right)$,
and
$\langle f \rangle \equiv (f,\phi_0) = \int_{\mathbf{T}^2}f(x,y)\,dxdy$
for the spatial average of $f$ on the torus.
We also denote by $\delta c \equiv c-c'$ the
difference between wave velocities.

\begin{thm}
\label{limitthm2}
For the statistical model (\ref{eq:generalmodel}) and 
(\ref{eq:discreteobservation}) with 
$\mathcal{L}=c\cdot \nabla$,
suppose that the data $Y_n'=\{ y_1',\cdots,y_n' \}$ is created
from (\ref{eq:discretetruthobservation}) with
$\mathcal{L}'=c'\cdot \nabla$ and that
$\delta c \ne 0 \mod (1,1)$
(equivalently
$\delta c \notin \mathbb{Z}\times \mathbb{Z}$).
As $n \to \infty$,
\begin{enumerate}
\item
if $\Delta t \,\delta c = (p'/p, q'/q) \in \mathbb{Q} \times \mathbb{Q}$
and $\gcd(p', p) = \gcd(q', q) = 1$,
then $m_n' \to \mathcal{F}_{(p,q)}u$
in the sense that
\begin{subequations}
\label{eq:discovmodelerror}
\begin{eqnarray}
&\left\| m_n'-\mathcal{F}_{(p,q)}u \right\|_{L^2(\Omega'; H^s(\mathbf{T}^2))}
= O\left(n^{-\frac{1}{2}}\right),
\label{eq:discovmodelerror-a}
\\
&\left\| m_n'-\mathcal{F}_{(p,q)}u \right\|_{H^s(\mathbf{T}^2)} = o\left(n^{-\theta}\right)
\quad \Omega'-a.s.,
\label{eq:discovmodelerror-b}
\end{eqnarray}
\end{subequations}
for any non-negative $\theta<1/2$;
\item
if $\Delta t \, \delta c \in \mathbb{R}\backslash \mathbb{Q} \times \mathbb{R} \backslash \mathbb{Q}$,
then $m_n' \to \langle u\rangle$
in the sense that
\begin{subequations}
\label{eq:discovmodelerror2}
\begin{eqnarray}
&\left\| m_n'-\langle u \rangle \right\|_{L^2(\Omega'; H^s(\mathbf{T}^2))}
= o\left(1\right),
\label{eq:discovmodelerror-c}
\\
&\left\| m_n'-\langle u \rangle \right\|_{H^s(\mathbf{T}^2)} = o\left(1\right)
\quad \Omega'-a.s.
\label{eq:discovmodelerror-d}
\end{eqnarray}
\end{subequations}
\end{enumerate}
\end{thm}
\begin{proof}
See \ref{app:B}.  
\end{proof}

\begin{remark}
It is interesting that it is not the {\em size} of 
$\Delta t \,\delta c$ 
but its {\em rationality or irrationality} which
determines the limit of $m_n'$.
This result may be understood intuitively
from Equation~(\ref{eq:averageadvection}), which reduces to
\begin{eqnarray}
  \label{eq:average}
\frac{1}{n}\sum_{l=0}^{n-1} u(\cdot + (l+1) \Delta t\,\delta c),
\end{eqnarray}
when $\mathcal{L}=c\cdot \nabla$ and $\mathcal{L}'=c'\cdot \nabla$.
It is possible to guess the large $n$ behaviour of Equation~(\ref{eq:average}) using the periodicity of $u$ and
an ergodicity argument.
The proof of Theorem~\ref{limitthm2} tells us that
the prediction resulting from this heuristic is indeed correct.

The proof of Corollary~\ref{almosteverwhere} tells us that,
whenever $s>1$,
the $H^s\left(\mathbf{T}^2\right)-$norm convergence
in (\ref{eq:discovmodelerror-b})
or (\ref{eq:discovmodelerror-d})
implies the almost everywhere convergence on $\mathbf{T}^2$ with the same order.
Therefore, 
$m_n' \to \mathcal{F}_{(p,q)}u$ or
$m_n' \to \langle u\rangle$
a.e. on $\mathbf{T}^2$ from Equation~(\ref{eq:discovmodelerror-b}) or
from Equation~(\ref{eq:discovmodelerror-d}),
when $\Delta t \,\delta c = (p'/p, q'/q) \in \mathbb{Q} \times \mathbb{Q}$
and $\gcd(p', p) = \gcd(q', q) = 1$, or
when $\Delta t \, \delta c \in \mathbb{R}\backslash \mathbb{Q} \times \mathbb{R} \backslash
\mathbb{Q}$,
respectively.
We will not repeat the statement of the corresponding result in the subsequent theorems.

When $\Delta t\,\delta c \in \mathbb{R}\backslash \mathbb{Q} \times \mathbb{R}\backslash
\mathbb{Q}$,
Equation~(\ref{eq:discovmodelerror2}) is obtained using
\begin{eqnarray}
\label{eq:ergodicity}
\frac{1}{n}\sum_{l=0}^{n-1} e^{2\pi i (k \cdot \delta c) t_{l+1} } = o(1),
\end{eqnarray}
as $n \to \infty$, from the theory of ergodicity
\cite{Breiman}.
The convergence rate of Equation~(\ref{eq:discovmodelerror2}) can be improved
if we have higher order convergence in Equation~(\ref{eq:ergodicity}).
It must be noted that in general there exists a fundamental relationship between the limits of $m_n'$
and the left-hand side of Equation~(\ref{eq:ergodicity}) for various $c$ and $c'$, as we will see.

Note also from Theorem~\ref{limitthm1} and Theorem~\ref{limitthm2},
the limit of $m_n'$ does not depend on the observation noise error $\Gamma \neq \Gamma'$
but does depend sensitively on the model error $\mathcal{L} \neq \mathcal{L}'$.
\qed
\end{remark}

\begin{corollary}
\label{limitwithmodelerror}
Under the same assumptions as Theorem~\ref{limitthm2},
as $n\to \infty$,
the $\Omega'-a.s.$ weak convergence $\mu_n' \Rightarrow \delta_{\mathcal{F}_{(p,q)}u}$ or
$\mu_n' \Rightarrow \delta_{\langle u\rangle}$ holds,
when $\Delta t \,\delta c = (p'/p, q'/q) \in \mathbb{Q} \times \mathbb{Q}$
and $\gcd(p', p) = \gcd(q', q) = 1$, or
when $\Delta t \, \delta c \in \mathbb{R}\backslash \mathbb{Q} \times \mathbb{R} \backslash \mathbb{Q}$,
respectively.
\end{corollary}
\begin{proof}
This is the same as the proof of Corollary~\ref{limitwithoutmodelerror},
so we omit it.
\end{proof}

\begin{remark}
\label{noncomm}
Our theorems show that
the observation accumulation limit and the vanishing model error limit
cannot be switched, i.e.,
\begin{eqnarray*}
\lim_{n\to \infty} \lim_{||\delta c||\to 0}
\left\| m_n'-u\right\|_{H^s(\mathbf{T}^2)}
& =  0, \\
\lim_{||\delta c||\to 0} \lim_{n\to \infty}
\left\| m_n'-u\right\|_{H^s(\mathbf{T}^2)}
& \neq  0 .
\end{eqnarray*}
Note the second limit is nonzero because $m_n'$ converges either to $\mathcal{F}_{(p,q)}u$
or $\langle u\rangle$.
\qed
\end{remark}

We can also study the effect of model error on the
filtering distribution, instead of the smoothing
distribution.
The following theorem extends Theorem~\ref{limitthm2}
to study the measure $(\mu')^n$,
showing that the truth $v_n'$ is not recovered if $\delta c \neq 0$.
This result should be compared with Theorem~\ref{NEW1} in the case of no model error.

\begin{thm}
\label{NEW2}
Under the same assumptions as Theorem~\ref{limitthm2},
as $n \to \infty$,
\begin{enumerate}
\item
if $\Delta t \,\delta c = (p'/p, q'/q) \in \mathbb{Q} \times \mathbb{Q}$
and $\gcd(p', p) = \gcd(q', q) = 1$,
then $(m')^n - \mathcal{F}_{(p,q)} e^{-t_n \mathcal{L}}u \to 0$
in the sense that
\begin{subequations}
\begin{eqnarray}
&\left\| (m')^n-\mathcal{F}_{(p,q)}e^{-t_n \mathcal{L}}u \right\|_{L^2(\Omega'; H^s(\mathbf{T}^2))}
= O\left(n^{-\frac{1}{2}}\right),
\\
&\left\| (m')^n-\mathcal{F}_{(p,q)}e^{-t_n \mathcal{L}}u
\right\|_{H^s(\mathbf{T}^2)} = o\left(n^{-\theta}\right)
\quad \Omega'-a.s.,
\end{eqnarray}
\end{subequations}
for any non-negative $\theta<1/2$;
\item
if $\Delta t \, \delta c \in \mathbb{R}\backslash \mathbb{Q} \times \mathbb{R} \backslash \mathbb{Q}$,
then $(m')^n \to \langle u\rangle$
in the sense that
\begin{subequations}
\label{eq:discovmodelerrorspaticalaverage}
\begin{eqnarray}
&\left\| (m')^n-\langle u \rangle \right\|_{L^2(\Omega'; H^s(\mathbf{T}^2))}
= o\left(1\right),
\\
&\left\| (m')^n-\langle u \rangle \right\|_{H^s(\mathbf{T}^2)} = o\left(1\right)
\quad \Omega'-a.s.
\end{eqnarray}
\end{subequations}
\end{enumerate}
\end{thm}
\begin{proof}
This is the same as the proof of Theorem~\ref{NEW1}
except
$\mathcal{F}_{(p,q)} e^{-t_n \mathcal{L}}u$
or $\langle u \rangle$ is used in place of $v_n'$,
so we omit it.
\end{proof}

\begin{corollary}
\label{limitwithmodelerror2}
Under the same assumptions as Theorem~\ref{limitthm2},
as $n\to \infty$,
the $\Omega'-a.s.$ weak convergence
$(\mu')^n - \delta_{\mathcal{F}_{(p,q)}e^{-t_n \mathcal{L}}u} \Rightarrow 0$ or
$(\mu')^n - \delta_{\langle u\rangle}\Rightarrow 0$ holds,
when $\Delta t \,\delta c = (p'/p, q'/q) \in \mathbb{Q} \times \mathbb{Q}$
and $\gcd(p', p) = \gcd(q', q) = 1$, or
when $\Delta t \, \delta c \in \mathbb{R}\backslash \mathbb{Q} \times \mathbb{R} \backslash \mathbb{Q}$,
respectively.
\end{corollary}
\begin{proof}
This is the same as the proof of Corollary~\ref{limitwithoutmodelerror},
so we omit it.
\end{proof}

Theorem~\ref{limitthm1} and Theorem~\ref{NEW1} show that, in the perfect model scenario,
the smoothing distribution on the initial condition and filtering distribution recover
the true initial condition and true signal, respectively,
even if the statistical model fails to capture the genuine covariance structure of the data.
Theorem~\ref{limitthm2} and Theorem~\ref{NEW2} show that the smoothing distribution
on the initial condition, and the filtering distribution,
do not converge to the truth, in the large data limit, when the model error corresponds 
to a constant shift in wave velocity,
however small.
In this case the wave velocity difference causes an advection
in Equation~(\ref{eq:averageadvection})
leading to recovery of only part of the 
Fourier expansion of $u$ as a limit of $m_n'$.
The next section concerns time-dependent model error
in the wave velocity,
and includes a situation intermediate
between those considered in this section. In
particular, a situation where the smoothing distribution is not
recovered correctly, but the filtering distribution is.

\section{Time-Dependent Model Error}
\label{tdme}
In the previous section
we studied model error for autonomous problems
where the operators $\mathcal{L}$ and $\mathcal{L}'$
(and hence $c$ and $c'$) are assumed time-independent.
However, our approach can be generalized to situations where
both operators are time-dependent: 
$\mathcal{L}(t)=c(t)\cdot \nabla$ and $\mathcal{L}'(t)=c'(t)\cdot \nabla$.
To this end, this section is devoted two problems
where
both the statistical model and the data are generated by
the non-autonomous dynamics. In the
first case deterministic dynamics with 
$c(t) -c'(t) \to 0$ (Theorem~\ref{limitthm3});
and in the second case where
the data is generated by the non-autonomous random dynamics with
$\mathcal{L}'(t;\omega') = c'(t;\omega') \cdot \nabla$
and $c'(t;\omega')$ being a random function fluctuating
around $c$
(Theorem~\ref{limitthm4}).
Here $\omega'$ denotes an element in the probability space that
generates $c'(t;\omega')$ and $\eta_n'$, assumed independent.

Now the statistical model~(\ref{eq:generalmodel}) becomes
\begin{eqnarray}
\label{eq:specificmodel2}
\partial_t v+ c(t) \cdot \nabla v =0,
\quad v(x, 0)=v_0(x)
\end{eqnarray}
Unless $c(t)$ is constant in time, 
the operator $e^{-t\mathcal{L}}$
is not a semigroup operator but for convenience we still
employ this notation
to represent the forward solution operator from time $0$ to time $t$ 
even for non-autonomous dynamics.
Then
the solution of Equation~(\ref{eq:specificmodel2})
is denoted by
\begin{eqnarray*}
v(x,t) = \left( e^{-t\mathcal{L}}v_0 \right) (x) \equiv v_0\left(x - \int^t_0 c(s)\,ds\right).
\end{eqnarray*}
This will correspond to a classical solution if
$v_0 \in C^1\left( \mathbf{T}^2, \mathbb{R} \right)$
and $c \in C\left( \mathbb{R}^+, \mathbb{R}^2 \right)$;
otherwise it will be a weak solution.
Similarly, the notation $e^{-t\mathcal{L}'}$ will be used for the forward solution operator
from time $0$ to time $t$
given the non-autonomous deterministic or random dynamics, i.e.,
Equation~(\ref{eq:generaltruthmodel}) becomes
\begin{eqnarray}
\label{eq:generaltruthmodelrandomwavespeed}
\partial_t v'+ c'(t;\omega') \cdot \nabla v' =0,
\quad
v'(0)=u
\end{eqnarray}
and we define the solution of Equation~(\ref{eq:generaltruthmodelrandomwavespeed}) by
\begin{eqnarray*}
v'(x,t;\omega')
= \left( e^{-t\mathcal{L}'} u \right) (x)
\equiv  u\left(x - \int^t_0 c'(s;\omega')\,ds\right),
\end{eqnarray*}
under the assumption that $\int^t_0 c'(s;\omega')\, ds$ is well-defined.
We will be particularly
interested in the case where $c'(t;\omega')$ is an affine function of a Brownian
white noise and then this expression corresponds to the 
Stratonovich solution of the PDE~(\ref{eq:generaltruthmodelrandomwavespeed})
\cite{flandoli2010well}.
Note the term $e^{t_{l+1} \left(\mathcal{L}-\mathcal{L}' \right)} u$
in Equation~(\ref{eq:posmean})
should be interpreted as
\begin{eqnarray*}
\left( e^{t_{l+1} \left(\mathcal{L}-\mathcal{L}' \right)} u \right)(x)
\equiv u\left(x + \int^t_0 \left( c(s) -c'(s;\omega') \right) \,ds\right).
\end{eqnarray*}

We now study the case where both $c(t)$ and $c'(t)$ are deterministic time-dependent wave velocities.
We here exhibit an intermediate situation between the two previously examined cases where
the smoothing distribution is not recovered correctly, 
but the filtering distribution is. 
This occurs when the
wave velocity
is time-dependent but converges in time to the true
wave velocity,
i.e., $\delta c(t) \equiv c(t)-c'(t) \to 0$, which is of interest
especially in view of Remark~\ref{noncomm}.
Let $u_\alpha \equiv u\left(\cdot + \alpha \right) $ be the translation of $u$ by $\alpha$.

\begin{thm}
\label{limitthm3}
For the statistical model (\ref{eq:generalmodel}) and (\ref{eq:discreteobservation}),
suppose that the data $Y_n'=\{ y_1',\cdots,y_n' \}$ is created
from (\ref{eq:discretetruthobservation}) with
$\mathcal{L}(t)=c(t)\cdot \nabla$ and $\mathcal{L}'(t)=c'(t)\cdot \nabla$
where $\delta c(t)=c(t)-c'(t)$
satisfies 
$\int^t_0 \delta c(s) \,ds =\alpha + {O}\left( t^{-\beta} \right)$.
Then,
as $n \to \infty$,
$m_n' \to u_\alpha$ in the sense that
\begin{subequations}
\label{eq:discovalpha}
\begin{eqnarray}
&\left\| m_n'-u_\alpha \right\|_{L^2(\Omega'; H^s(\mathbf{T}^2))} = O\left(n^{-\phi}\right),
\label{eq:discovalpha-a} \\
&\left\| m_n'-u_\alpha \right\|_{H^s(\mathbf{T}^2)} = o\left(n^{-\theta}\right)
\quad \Omega'-a.s.,
\label{eq:discovalpha-b}
\end{eqnarray}
\end{subequations}
for $\phi=1/2 \wedge \beta$ and
for any non-negative $\theta < \phi$.  
\end{thm}
\begin{proof}
See \ref{app:B}.
\end{proof}

\begin{thm}
\label{NEW3}
Under the same assumptions as Theorem~\ref{limitthm3},
if $u$ is Lipschitz continuous in $H^s(\mathbf{T}^2)$
where $s$ is given in Assumptions~\ref{ass1},
then $(m')^n-v'_n \to 0$ in the sense that
\begin{subequations}
\label{eq:disfilter}
\begin{eqnarray}
&\left\| (m')^n-v'_n \right\|_{L^2(\Omega'; H^s(\mathbf{T}^2))}
= O\left(n^{-\phi}\right),
\label{eq:disfilter-a} \\
&\left\| (m')^n-v'_n \right\|_{H^s(\mathbf{T}^2)} = o\left(n^{-\theta}\right)
\quad \Omega'-a.s.,
\label{eq:disfilter-b}
\end{eqnarray}
\end{subequations}
for $\phi=1/2 \wedge \beta$ and
for any non-negative $\theta < \phi$.  
\end{thm}

\begin{proof}
Equation~(\ref{eq:disfilter-b}) follows from
\begin{eqnarray*}
\fl \left\| (m')^n -v'_n \right\|_{H^s(\mathbf{T}^2)}
&= \left\| e^{-t_n\mathcal{L}}(m_n' -u_\alpha)
+ \left( e^{-t_n \mathcal{L}}u_\alpha - e^{-t_n \mathcal{L}'}u \right) \right\|_{H^s(\mathbf{T}^2)} \\
& \leq \left\| e^{-t_n\mathcal{L}} \right\|_{\mathbf{L}(H^s(\mathbf{T}^2);H^s(\mathbf{T}^2))}
\left\| m_n' - u_\alpha \right\|_{H^s(\mathbf{T}^2)}\\
& \quad + \left\| u\left(\cdot+\alpha - \int^{t_n}_0 c(s)\,ds \right)
-u\left( \cdot - \int^{t_n}_0 c'(s)\,ds  \right) \right\|_{H^s(\mathbf{T}^2)} \\
& \leq \left\| m_n' - u_\alpha \right\|_{H^s(\mathbf{T}^2)}
+ C \left\| \alpha- \int^{t_n}_0 \delta c(s)\,ds \right\|,
\end{eqnarray*}
and Equation~(\ref{eq:disfilter-a}) follows from
\begin{eqnarray*}
& \left\| (m')^n-v'_n \right\|^2_{L^2(\Omega'; H^s(\mathbf{T}^2))} \\
& \quad\quad  = \mathbb{E} \left\| (m')^n -v'_n \right\|_{H^s(\mathbf{T}^2)}^2 \\
& \quad\quad  = \mathbb{E} \left\| e^{-t_n\mathcal{L}}\left(m_n'-u_\alpha \right)
+ \left(e^{-t_n\mathcal{L}}u_\alpha - e^{-t_n\mathcal{L}'}u\right) \right\|_{H^s(\mathbf{T}^2)}^2 \\
& \quad\quad \leq 2\left\| e^{-t_n\mathcal{L}}
\right\|^2_{\mathbf{L}(H^s(\mathbf{T}^2);H^s(\mathbf{T}^2))}
 \mathbb{E} \left\| m_n' - u_\alpha \right\|_{H^s(\mathbf{T}^2)}^2 \\
& \quad\quad \quad +2 \left\| u\left(\cdot+\alpha
- \int^{t_n}_0 c(s)\,ds  \right)
-u\left( \cdot - \int^{t_n}_0 c'(s)\,ds \right) \right\|_{H^s(\mathbf{T}^2)}^2.
\end{eqnarray*}
\end{proof}

Finally, we study the case where $c(t)$ is deterministic
but $c'(t;\omega')$ is a random process.
We here note that while the true signal solves a
linear SPDE with multiplicative noise,
Equation~(\ref{eq:generaltruthmodelrandomwavespeed}),
the statistical model used to filter is a linear deterministic
PDE, Equation~(\ref{eq:specificmodel2}).
We study the specific case
$c'(t;\omega') =  c(t) -  \varepsilon \dot{W}(t)$, i.e.,
the deterministic wave velocity 
is modulated by a white noise
with small amplitude $\varepsilon > 0$.

\begin{thm}
\label{limitthm4}
For the statistical model (\ref{eq:generalmodel}) and (\ref{eq:discreteobservation}),
suppose that the data $Y_n'=\{ y_1',\cdots,y_n' \}$ is created
from (\ref{eq:discretetruthobservation}) with
$\mathcal{L}(t)=c(t)\cdot \nabla$ and $\mathcal{L}'(t;\omega')=c'(t;\omega')\cdot \nabla$
where $\int^t_0 c'(s;\omega')\,ds = \int^t_0 c(t)\,ds - \varepsilon {W}(t)$
and $\varepsilon W(t)$ is the Wiener process with amplitude $\varepsilon >0$.
Then,
as $n \to \infty$,
$m_n' \to \langle u\rangle$ in the sense that
\begin{subequations}
  \label{eq:limitthm4eq}
\begin{eqnarray}
&\left\| m_n'-\langle u\rangle \right\|_{L^2(\Omega'; H^s(\mathbf{T}^2))} =
O\left(n^{-\frac{1}{2}}\right),
\\
&\left\| m_n'-\langle u\rangle \right\|_{H^s(\mathbf{T}^2)} = o\left(n^{-\theta}\right)
\quad \Omega'-a.s.,
\end{eqnarray}
\end{subequations}
for any non-negative $\theta <1/2$.  
\end{thm}
\begin{proof}
See \ref{app:B}.
\end{proof}

We do not state the corresponding theorem on the mean of the filtering distribution
as $(m')^n$ converges to the same constant $\langle u\rangle$ with the same rate 
shown in Equations~(\ref{eq:limitthm4eq}).

\begin{remark}
  \label{generallimit}
In Theorem~\ref{limitthm4}
the law of $\varepsilon W(t)$ converges weakly to 
a uniform distribution on the torus
by the L\'{e}vy-Cram\'{e}r
continuity theorem~\cite{Varadhan}.
The limit $\langle u\rangle$ is the average of $u$
with respect to this measure.
This result may be generalized to consider the
case where 
$\int^t_0 \left( c(s) - c'(s;\omega') \right) \,ds$ 
converges weakly to the measure $\nu$ as $t \to \infty$. 
Then $m_n'(\cdot) \to \int u(\cdot+y)\,d\nu(y)$ as $n \to \infty$
in the same norms as used in Theorem~\ref{limitthm4} but,
in general, with no rate. 
The key fact underlying the result is
\begin{eqnarray}
\label{eq:sLLN}
\frac{1}{n}\sum_{l=0}^{n-1}
e^{2\pi i k \cdot \int^{t_{l+1}}_0 \left( c(s) - c'(s;\omega') \right) \,ds}
-
\int u(\cdot+y)\,d\nu(y)
= o(1)
\quad \Omega'-a.s.,
\end{eqnarray}
as $n \to \infty$,
which follows
from the strong law of large numbers 
and the Kolmogorov zero-one law~\cite{Varadhan}.
Depending on the process
$\int^t_0 \left( c(s) - c'(s;\omega') \right) \,ds$,
an algebraic convergence rate can be determined 
if we have higher order convergence result for the corresponding Equation~(\ref{eq:sLLN}).
This viewpoint may be used to provide a common framework
for all the limit theorems in this paper.
\end{remark}

\section{Numerical Illustrations}
\label{BIPsection}
\subsection{Algorithmic Setting}
The purpose of this section is twofold: first to
illustrate the preceding theorems with numerical experiments;
and secondly to show that relaxing the statistical
model can avoid some of the lack of consistency problems
that the theorems highlight. All of the numerical
results we describe are based on using the
Equations~(\ref{eq:generalmodel}), (\ref{eq:discreteobservation}),
with ${\mathcal L}=c\cdot\nabla$ for some constant 
wave velocity $c$, so that the
underlying dynamics is given by 
Equation~(\ref{eq:specificmodel}).
The data is generated by (\ref{eq:generaltruthmodel}),
(\ref{eq:discretetruthobservation}) with
${\mathcal L}'=c'(t)\cdot\nabla$, for a variety of choices
of $c'(t)$ (possibly random), 
and in subsection \ref{SIC} we
illustrate Theorems \ref{limitthm1}, 
\ref{limitthm2}, \ref{limitthm3} and \ref{limitthm4}.
In subsection \ref{SWS} we will also describe a 
numerical method in which
the state of the system {\em and} the wave velocity are
learnt by combining the data and statistical model.
Since this problem is inherently non-Gaussian we
adopt from the outset a Bayesian approach 
which coincides with the (Gaussian) filtering or smoothing
approach when the wave velocity is fixed, but is sufficiently 
general to also allow for the wave velocity to be part of the 
unknown state of the system.
In both cases we apply functionspace MCMC methods \cite{Acta10}
to sample the distribution of interest.
Note, however, that the purpose of this section is {\em not}
to determine the most efficient numerical methods, but
rather to study the properties of the statistical
distributions of interest.

For fixed wave velocity $c$ the statistical model 
(\ref{eq:generalmodel}), (\ref{eq:discreteobservation})
defines a probability distribution 
$\mathbb{P}(v_0, Y_n \vert c).$
This is a Gaussian distribution and the conditional
distribution $\mathbb{P}(v_0 \vert Y_n,c)$ is given
by the measure $\mu_n={\mathcal N}(m_n,{\mathcal C}_n)$ 
studied in sections \ref{Kalman}, \ref{MCTI} and
\ref{tdme}. 
In our first set of numerical results, in subsection \ref{SIC},
the wave velocity is considered known.  We sample
$\mathbb{P}(v_0 \vert Y_n,c)$ using the functionspace
random walk from \cite{CDS10}. 
In the second set of results, in subsection \ref{SWS},
the wave velocity is considered as an unknown constant. 
If we place a prior measure $\rho(c)$ 
on the wave velocity then we may define 
$\mathbb{P}(c,v_0,Y_n)=\mathbb{P}(v_0, Y_n \vert c)\rho(c).$
We are then interested in the conditional distribution
$\mathbb{P}(c,v_0 \vert Y_n)$ which is non-Gaussian.
We adopt a Metropolis-within-Gibbs 
approach~\cite{geweke2001bayesian, roberts2006harris}
in which we sample alternately
from $\mathbb{P}(v_0 \vert  c,Y_n)$, 
which we do as in subsection \ref{SIC},
and $\mathbb{P}(c \vert v_0, Y_n)$, which we sample
using a finite dimensional Metropolis-Hastings method.

Throughout the numerical simulations we represent
the solution of the wave equation on a grid of $2^5
\times 2^5$ points, 
and observations are also taken on this grid. 
The observational noise is white (uncorrelated)
with variance $\sigma^2=10^{-4}$ at each grid point.
The continuum limit of such a covariance operator 
does not satisfy  Assumptions~\ref{ass1}, but is
used to illustrate the fact that the theoretical results
can be generalized to such observations. Note also that
the numerical results are performed with model error
so that the aforementioned distributions are sampled 
with $Y_n=Y_n'$ from (\ref{eq:generaltruthmodel}),
(\ref{eq:discretetruthobservation}).

\subsection{Sampling the initial condition with model error}
\label{SIC}
Throughout we use the wave velocity
\begin{eqnarray}
\label{eq:truec}
c=(-0.5, -1.0),
\end{eqnarray}
in our statistical model.
The true initial condition used to generate the data is 
\begin{eqnarray}
u(x_1,x_2) = \sum_{k_1,k_2=1}^3 \sin(2 \pi k_1 x_1) + \cos(2 \pi k_2 x_2).
\label{eq:trueic}
\end{eqnarray}
This function is displayed in Figure~\ref{fig:u}.
As prior on $v_0$ we choose the Gaussian
$\mathcal{N}\left( 0, (-\triangle)^{-2} \right)$ 
where the domain of $-\triangle$ is $H^2(\mathbf{T}^2)$
with constants removed, so that it is positive.
We implement the MCMC method to sample 
from $\mathbb{P}(v_0 | c,Y_n =Y_n')$ 
for a number of different data $Y_n'$, corresponding to
different choices of $c'=c'(t,\omega')$. We calculate 
the empirical mean of $\mathbb{P}(v_0 | c,Y_n =Y_n')$,
which approximates $\mathbb{E}(v_0 | c, Y_n =Y_n').$
The results are shown in
Figures~\ref{fig:postmean}--\ref{fig:postmean_dc_noisy}.
In all cases the Markov chain is
burnt in for $10^6$ iterations, and this
transient part of the simulation is not used to compute draws
from the conditioned measure $\mathbb{P}(v_0 | c,Y_n =Y_n')$.
After the burn in we proceed
to iterate a further $10^7$ times and use this
information to compute the corresponding moments.
The data size $n$ is chosen sufficiently large that this
distribution is approximately a Dirac measure.

In the perfect model scenario ($c=c'$),
the empirical mean 
shown in Figure~\ref{fig:postmean},
should fully recover the true initial condition $u$
from Theorem~\ref{limitthm1}.
Comparison with Figure~\ref{fig:u} shows that this
is indeed the case, 
illustrating Corollary~\ref{almosteverwhere}.
We now demonstrate the effect of model error in the
form of a constant shift in the wave velocity: 
Figure~\ref{fig:postmean_dc0p5} and Figure~\ref{fig:postmean_dc_irr}
show the empirical means when $c-c' = (1/2, 1/2) \in \mathbb{Q} \times \mathbb{Q}$
and $c-c' =(1/e, 1/\pi) \in \mathbb{R}\backslash \mathbb{Q} \times \mathbb{R} \backslash \mathbb{Q}$,
respectively.
From Theorem~\ref{limitthm2}, the
computed empirical distribution should be close to, respectively, 
$\mathcal{F}_{(2,2)}u$ comprising only the mode 
$(k_1,k_2)=(2,2)$ from (\ref{eq:trueic}), 
or $\langle u\rangle =0$; this is indeed the case.

If we choose $c'(t)$ satisfying $\int_{0}^{\infty} \left( c-c'(s) \right) \,ds = (1/2, 1/2)$,
then Theorem~\ref{limitthm3} tells us that
Figure~\ref{fig:postmean_int_dc0p5} should be close to
a shift of $u$ by $(1/2,1/2)$, and this is exactly
what we observe.
In this case, we know from Theorem~\ref{NEW3}
that although the smoother is in error,
the filter should correctly recover the true $v_n'$ for large
$n$.  To illustrate this we compute
$\| \mathbb{E}(v_n | c,Y_n=Y_n') -v_n' \|_{L^2(\mathbf{T}^2)}$
as a function of $n$ and depict it in
Figure~\ref{fig:filterimprove}.
This shows convergence to $0$ as predicted.
To obtain a rate of convergence,
we compute the gradient of a $\log$-$\log$ plot of
Figure~\ref{fig:filterrate}.
We observe the rate of convergence is close to
$O(n^{-2})$. 
Note that this is higher than the theoretical bound of
$O(n^{-\phi})$,
with $\phi=1/2 \wedge \beta$, given in Equation~(\ref{eq:disfilter-b});
this suggests that our convergence theorems do not have sharp rates.

Finally, we examine the random $c'(t,\omega')$ cases.
Figure~\ref{fig:postmean_dc_noisy} shows the
empirical mean when
$c'(t;\omega')$ is chosen such that
$$\int_{0}^{t} \left( c-c'(s;\omega') \right) \,ds = W(t)$$
where $W(t)$ is a standard Brownian motion.
Theorem~\ref{limitthm4} tells us that the
computed empirical distribution 
should have mean close to $\langle u \rangle$,
and this is again the case.

\setcounter{subfigure}{-1}
\begin{figure}[htp]
  \centering
  \subfigure{
    \includegraphics[type=eps, ext=.eps, read=.eps, width=\textwidth]{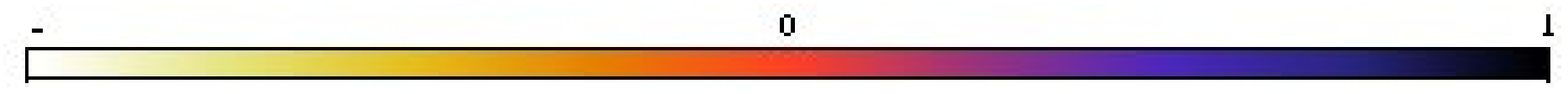}}
  \subfigure[$u(x_1,x_2)$]{
    \label{fig:u}
    \includegraphics[type=eps, ext=.eps, read=.eps, width=0.4\textwidth]{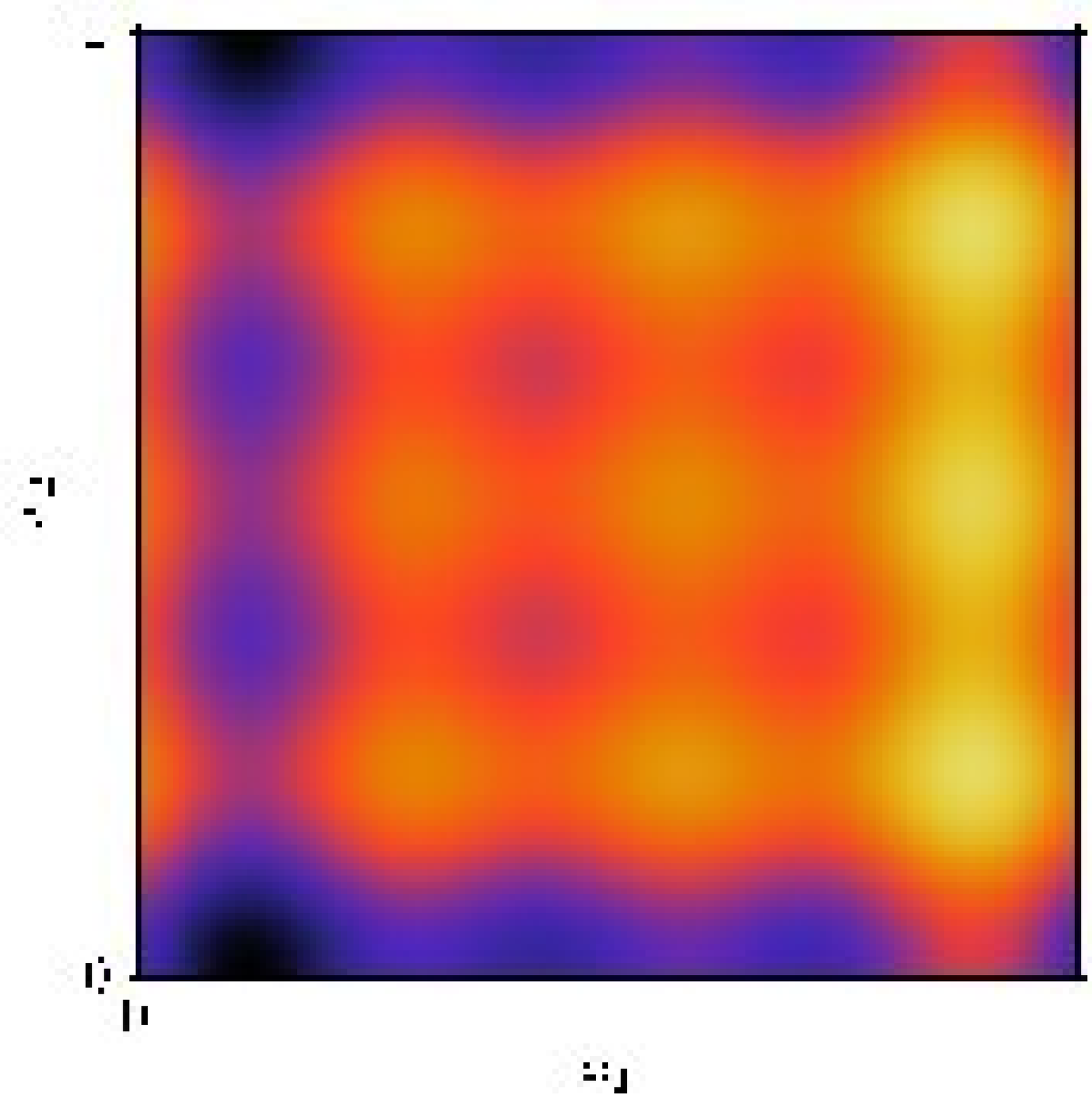}}
  \subfigure[$\delta c = (0, 0)$]{
    \label{fig:postmean}
    \includegraphics[type=eps, ext=.eps, read=.eps, width=0.4\textwidth]{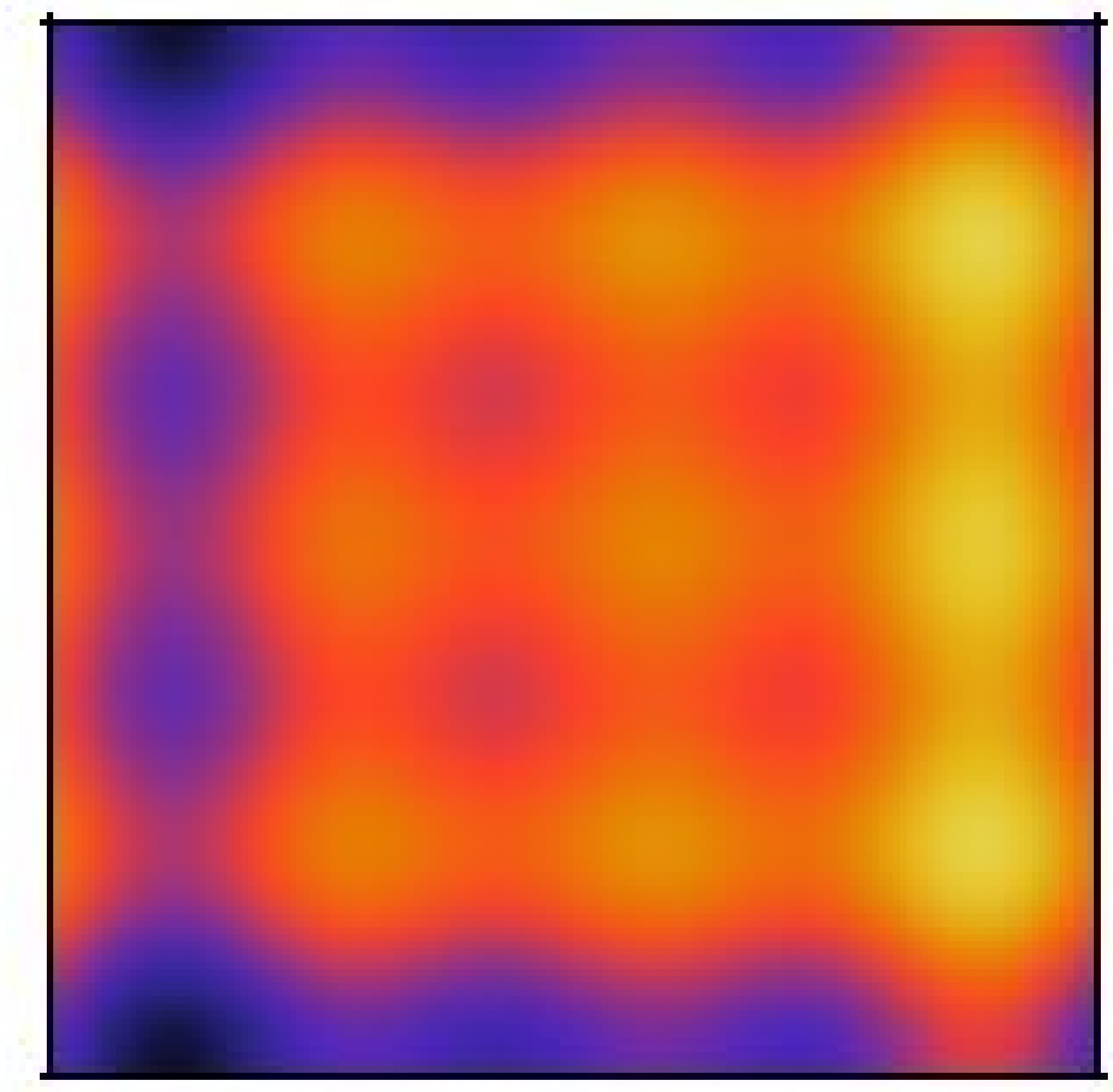}} \\
  \subfigure[$\delta c = (1/2, 1/2)$]{
    \label{fig:postmean_dc0p5}
    \includegraphics[type=eps, ext=.eps, read=.eps, width=0.4\textwidth]{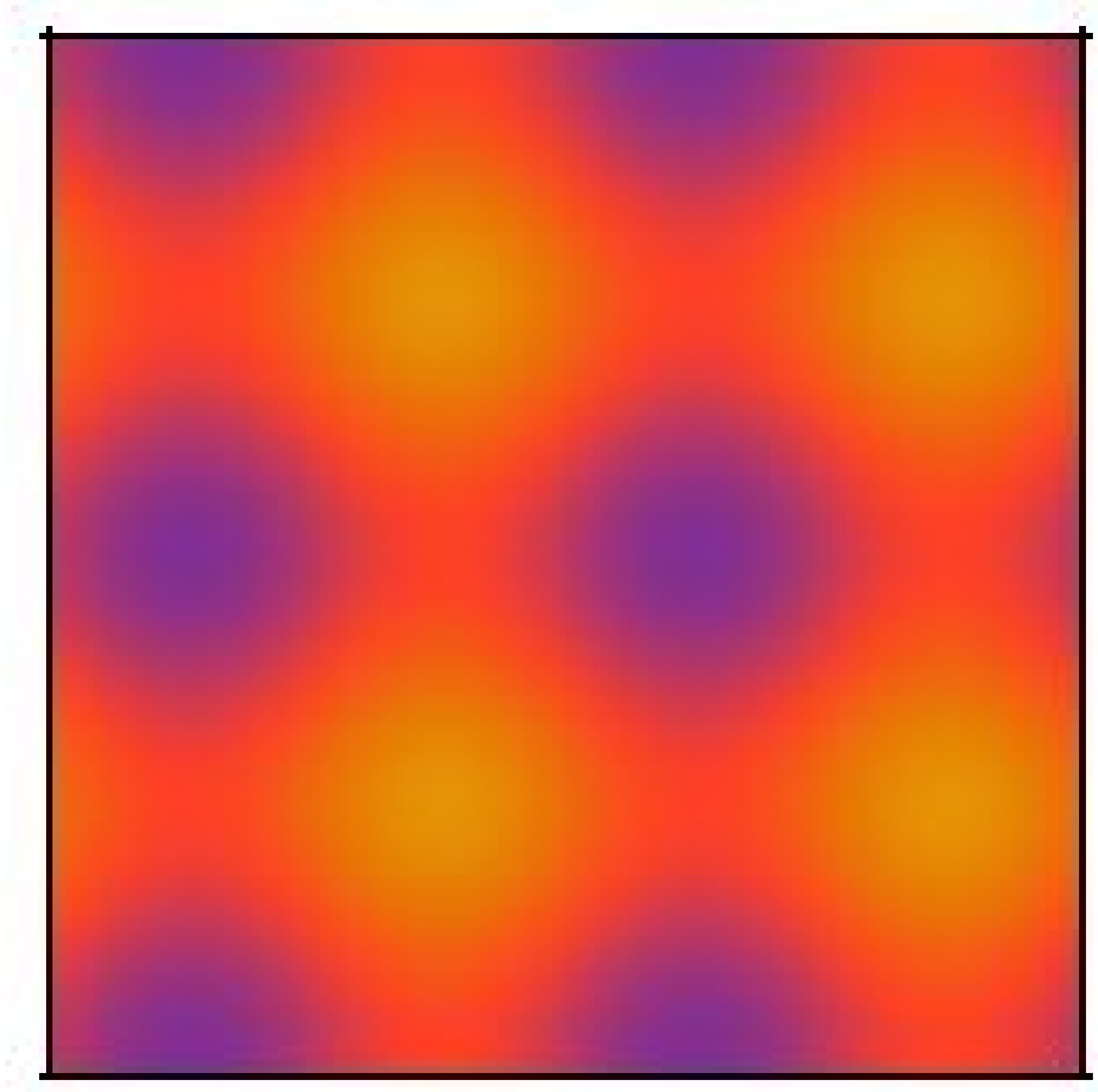}}
  \subfigure[
  	$\delta c = (1/e, 1/\pi)
  	$
  	]{
    \label{fig:postmean_dc_irr}
    \includegraphics[type=eps, ext=.eps, read=.eps, width=0.4\textwidth]{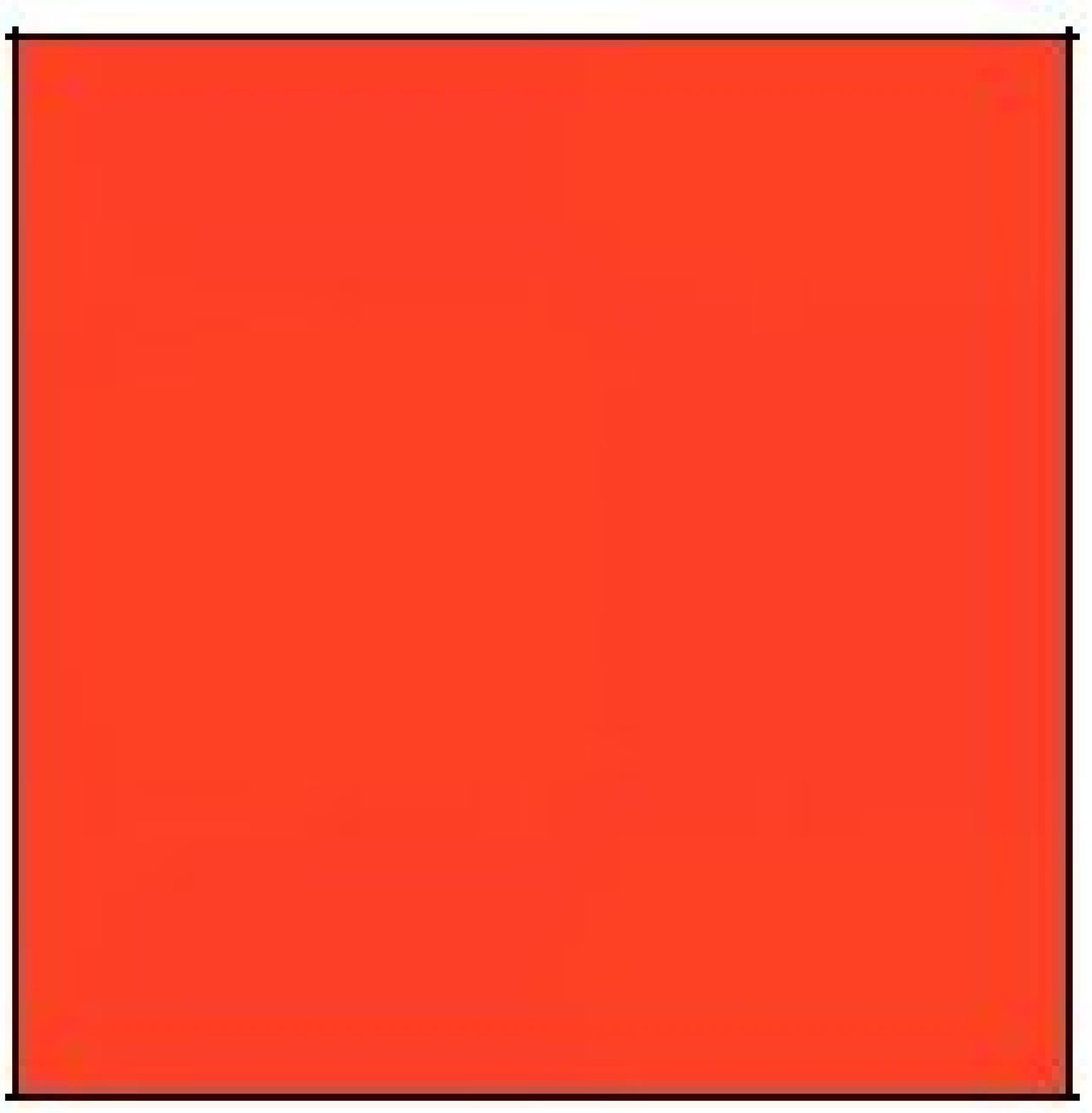}} \\
  \subfigure[$\int_0^{\infty} (c-c'(s)) \,ds = (1/2, 1/2)$]{
    \label{fig:postmean_int_dc0p5}
    \includegraphics[type=eps, ext=.eps, read=.eps, width=0.4\textwidth]{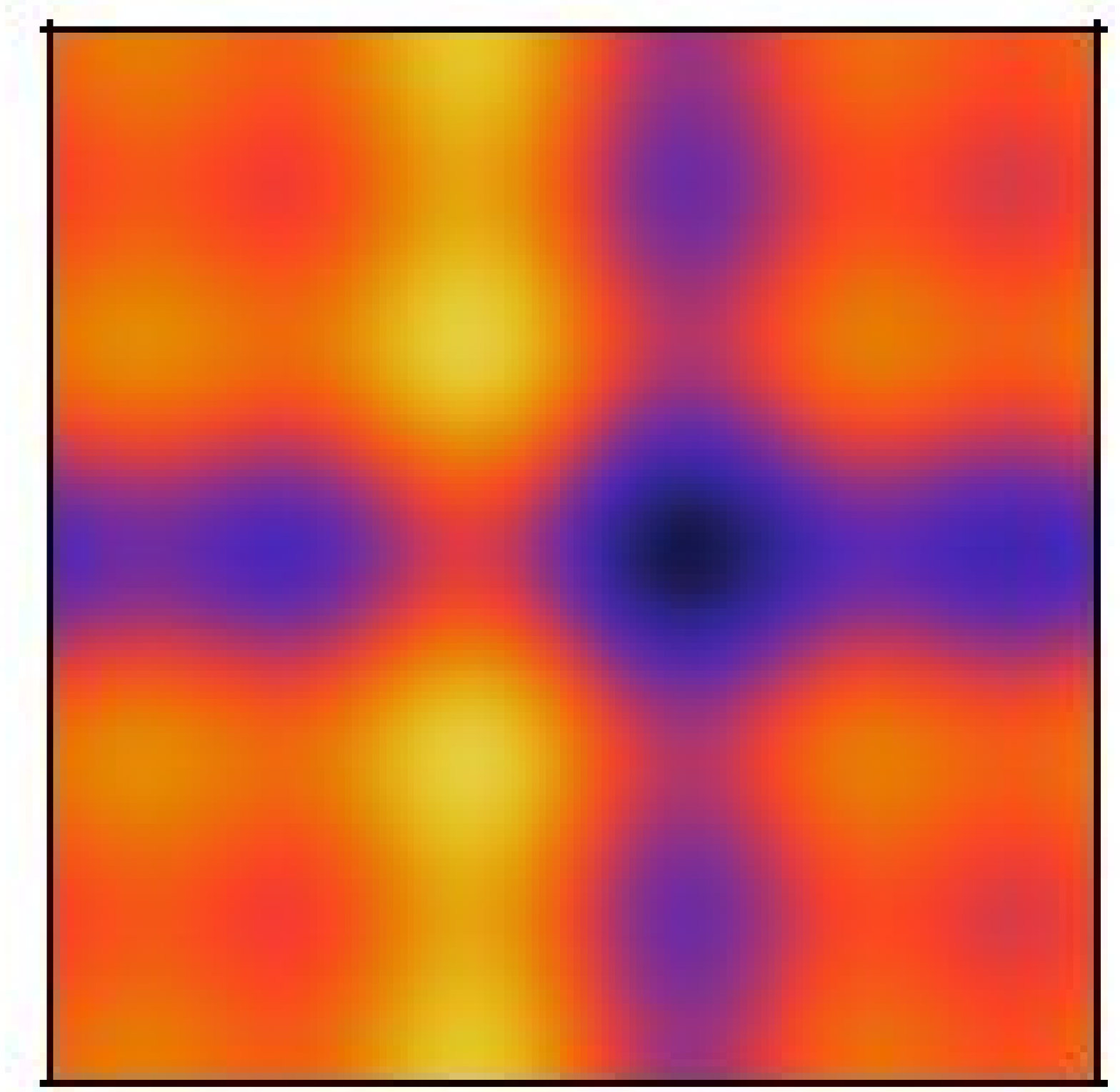}}
  \subfigure[
	$\int_{0}^{t} \left( c-c'(s;\omega') \right) \,ds = W(t)$
  ]{
    \label{fig:postmean_dc_noisy}
    \includegraphics[type=eps, ext=.eps, read=.eps, width=0.4\textwidth]{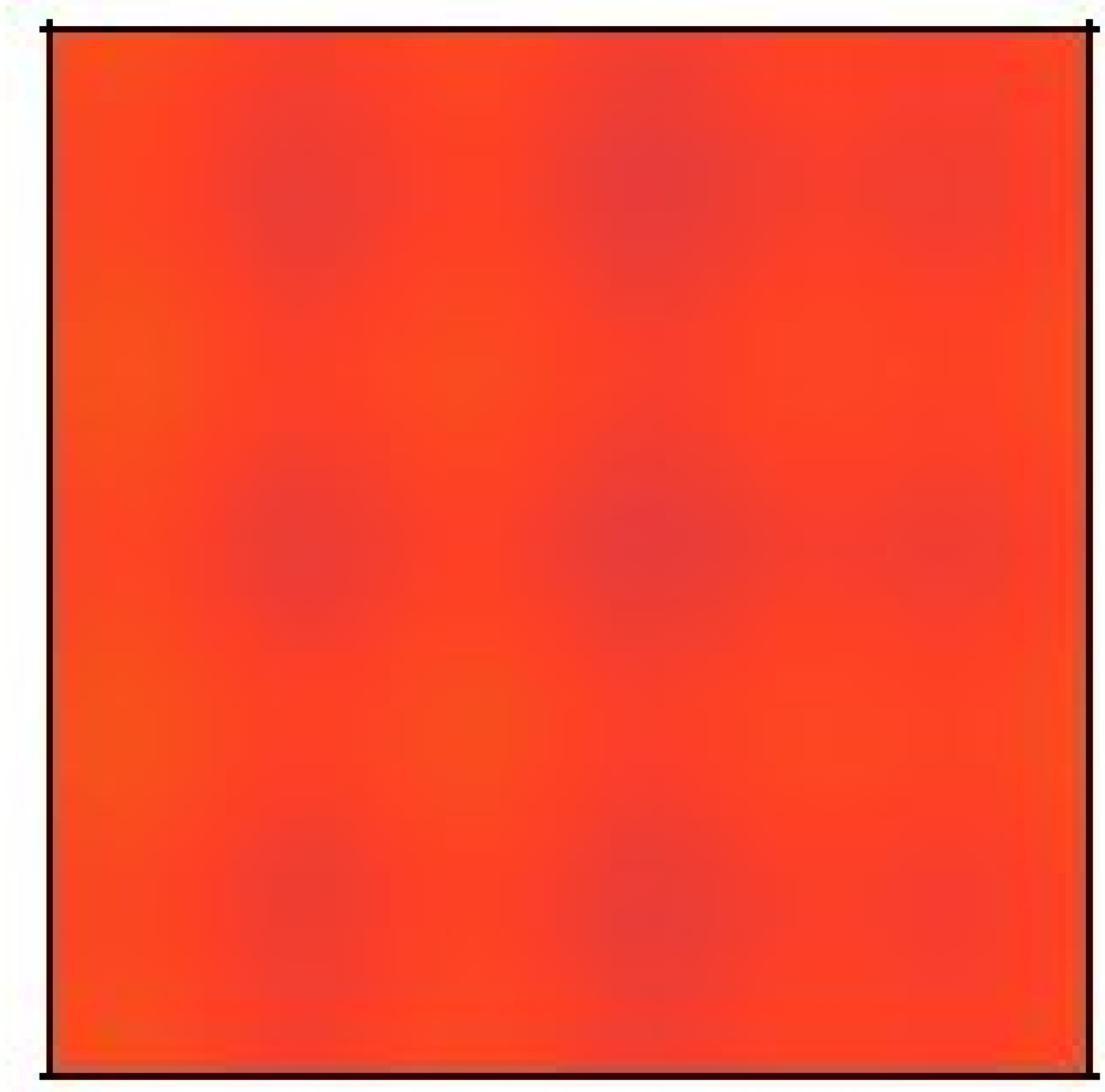}}
  \caption{
  Figure \ref{fig:u} is the true initial condition.
Figures \ref{fig:postmean} --
\ref{fig:postmean_dc_noisy} show the desired
empirical mean of
the smoothing
$\mathbb{P}(v_0 | Y_n=Y_n')$ 
for $\delta c = (0, 0),
\delta c = (1/2,1/2), \delta c \in \mathbb{R} \setminus \mathbb{Q} \times \mathbb{R} \setminus
\mathbb{Q}, \int_0^{\infty} \delta c \,dt = (1/2, 1/2)$ and $\delta c = \dot{W}$ respectively.
}
  \label{fig:main}
\end{figure}

\begin{figure}[htp]
  \centering
  \subfigure[]{
    \label{fig:filterimprove}
    \includegraphics[type=eps, ext=.eps, read=.eps, width=0.4\textwidth]{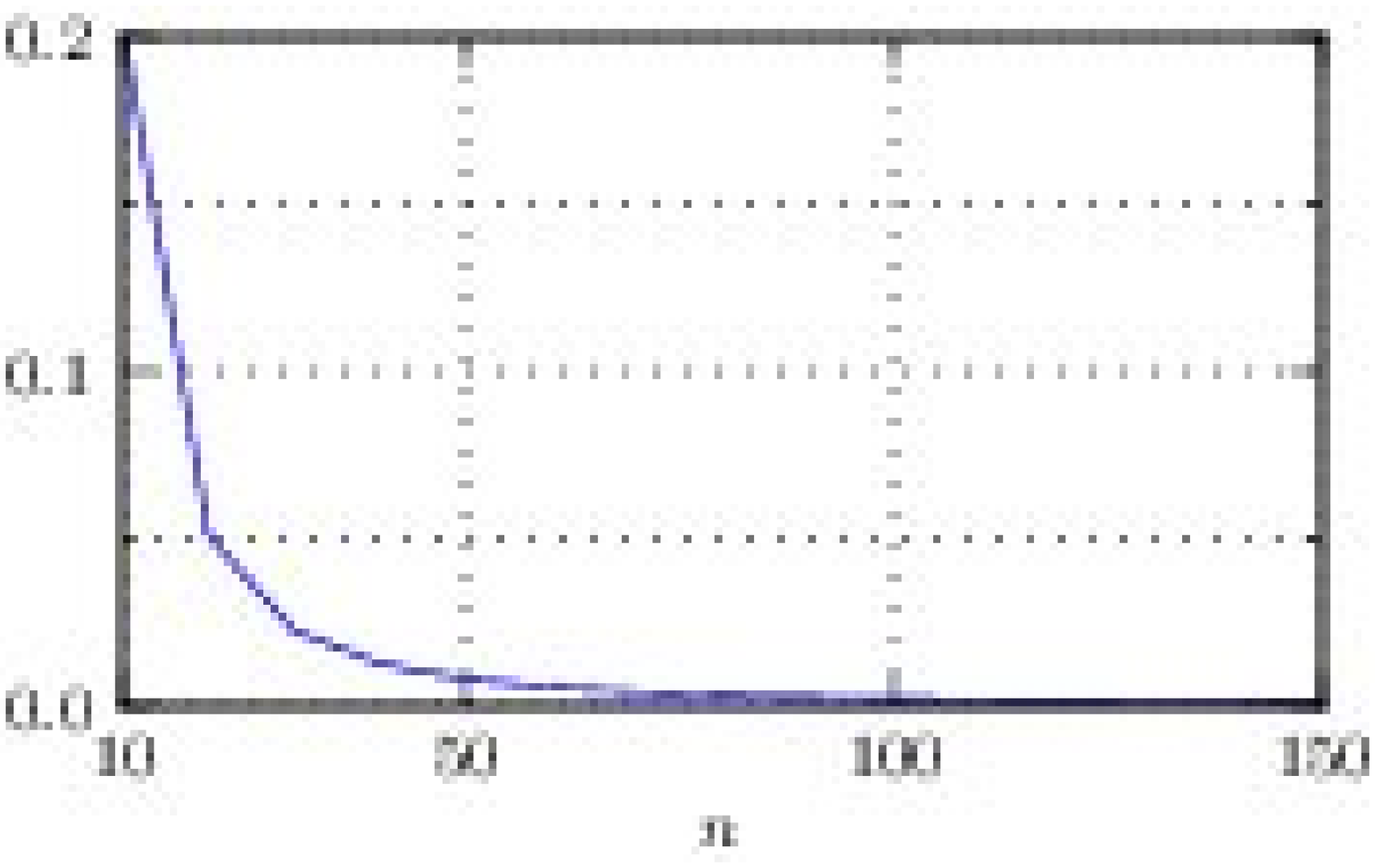}}
    \subfigure[]{
    \label{fig:filterrate}
    \includegraphics[type=eps, ext=.eps, read=.eps, width=0.4\textwidth]{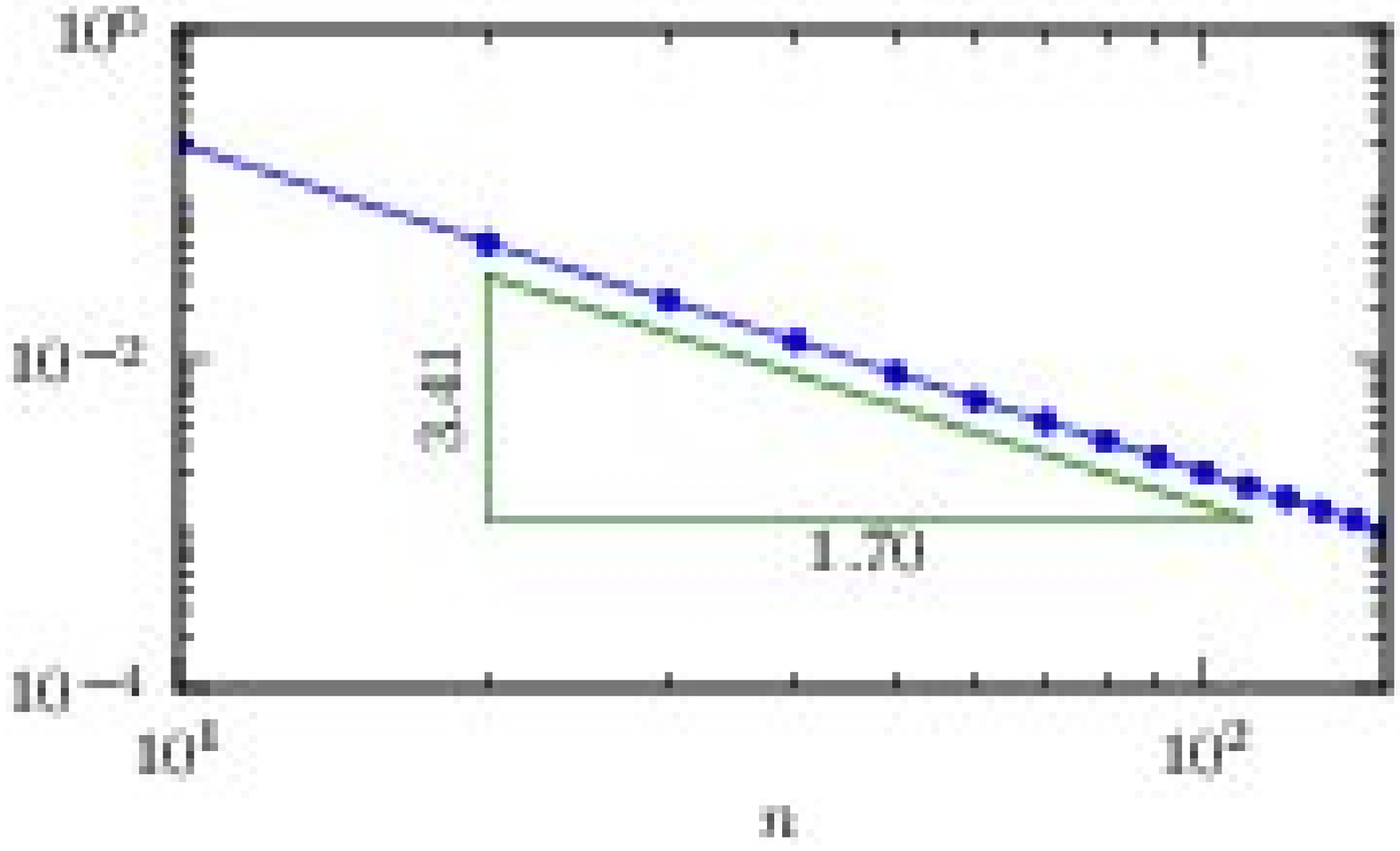}}
	\caption{
	Plot \ref{fig:filterimprove} shows $\| \mathbb{E} (v_n | c, Y_n=Y_n') - v'_{n}
	\|^2_{L^{2}(\mathbb{T}^2)}$ as a function of $n$, when $\int_0^{\infty} \delta c(s) \,ds = (1/2,
	1/2)$. Its $\log$-$\log$ plot, along with a least
squares fit, is depicted in Plot \ref{fig:filterrate}, demonstrating quadratic convergence.
}
\end{figure}

\subsection{Sampling the wave velocity and initial condition}
\label{SWS}

The objective of this subsection is to show that
the problems caused by model error in the
form of a constant shift to the wave velocity can be overcome
by sampling $c$ and $v_0$. We generate data from
(\ref{eq:generaltruthmodel}),
(\ref{eq:discretetruthobservation}) with
$c'=c$ given by (\ref{eq:truec}) and initial condition
(\ref{eq:trueic}). We assume that neither the wave velocity
nor the initial condition are known to us, and we
attempt to recover them from given data.

The desired conditional distribution 
is multimodal with respect to $c$ -- recall that
it is non-Gaussian --
and care is required to seed the chain close
to the desired value in order to avoid metastability.
Although the algorithm does not have access to
the true signal $v'_n,$ we do have noisy observations
of it: $y'_n$.  Thus it is natural to choose as initial 
$c=c^*$ for the Markov chain the value which minimizes
\begin{eqnarray}
\label{eq:seed}
\sum_{j=1}^{n-1} \left\| \log\left( \frac{(v'_{j+1},\phi_k)}{(v'_{j},\phi_k)} \right)
-  2\pi i k \cdot c \Delta t \right\|^2.
\end{eqnarray}
Because of the observational noise this estimate is
more accurate for small values of $k$ and we
choose
$k=(1,0)$
to estimate $c_1$ and $k=(0,1)$ to estimate $c_2$.

Figure~\ref{fig:c-samples} shows the marginal distribution
for $c$ computed with
four different values of the
data size $n$, in all cases with the Markov chain
seeded as in (\ref{eq:seed}).
The results show that the marginal 
wave velocity distribution 
$\mathbb{P}(c |Y_n)$ converges
to a Dirac on the true value as the amount of data is
increased. Although not shown here, the initial condition
is also converging to a Dirac on the true value
(\ref{eq:truec}) in this limit.

We round-off this subsection by mentioning related
published literature.
First we mention that, in a setting similar to ours,
a scheme to approximate the true wave velocity is
proposed which uses parameter estimation within 3D Var
for the linear advection equation  with constant velocity \cite{smith2009variational},
and its hybrid with the EnKF for the non-constant velocity case \cite{smith2010hybrid}.
These methodologies deal with the problem entirely in finite
dimensions but are not limited to the linear dynamics.
Secondly we note that, although a constant wave velocity
parameter in the linear advection equation is a useful
physical idealization in some cases, it is a very rigid
assumption, making the data assimilation
problem with respect to this parameter quite hard; 
this is manifest in the large number of
samples required to estimate this constant parameter. 
A notable, and desirable,
direction in which to extend this work numerically is to consider
the time-dependent wave velocity as presented in Theorems
\ref{limitthm3}--\ref{limitthm4}. For efficient filtering techniques
to estimate time-dependent parameters, the reader is directed to
\cite{cohn1997introduction, dee1998data, baek2006local, gershgorin2010test}.

\begin{figure}[htp]
  \centering
  \subfigure[
  $n=10$
  ]{
    \label{fig:csamples_10}
    \includegraphics[type=eps, ext=.eps, read=.eps, width=0.4\textwidth]{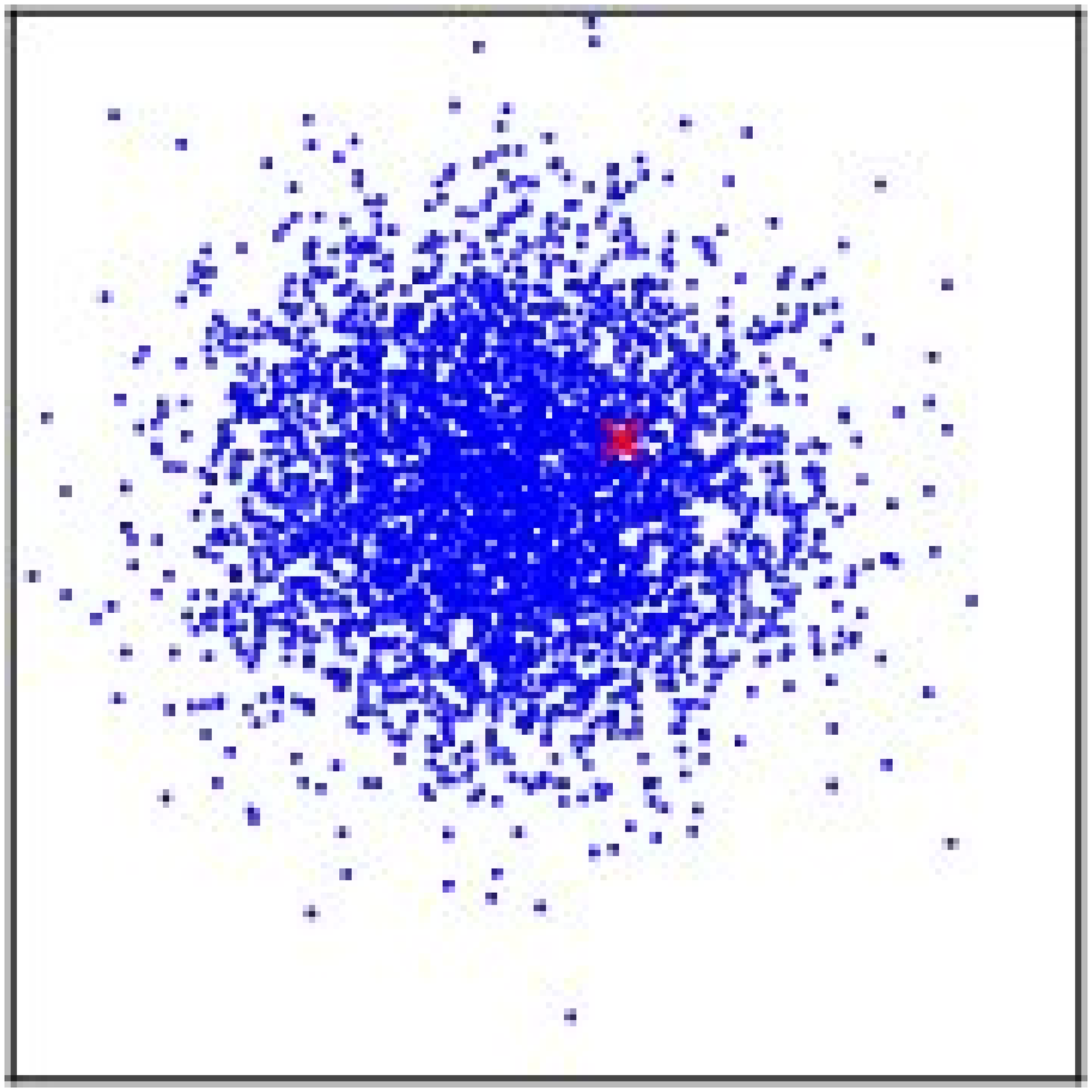}}
  \subfigure[
  $n=50$
  ]{
    \label{fig:csamples_50}
    \includegraphics[type=eps, ext=.eps, read=.eps, width=0.4\textwidth]{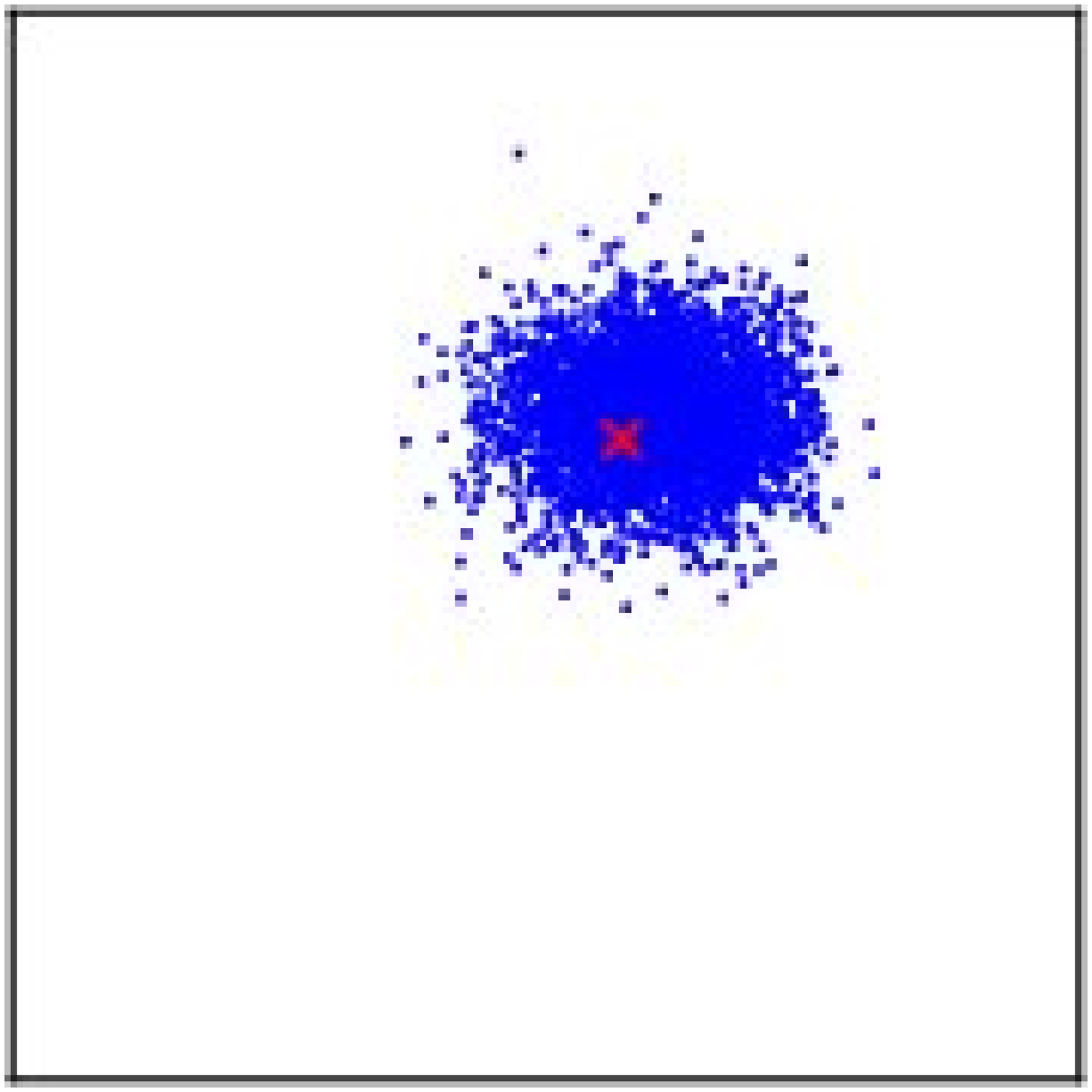}} \\
  \subfigure[
  $n=100$
  ]{
    \label{fig:csamples_100}
    \includegraphics[type=eps, ext=.eps, read=.eps, width=0.4\textwidth]{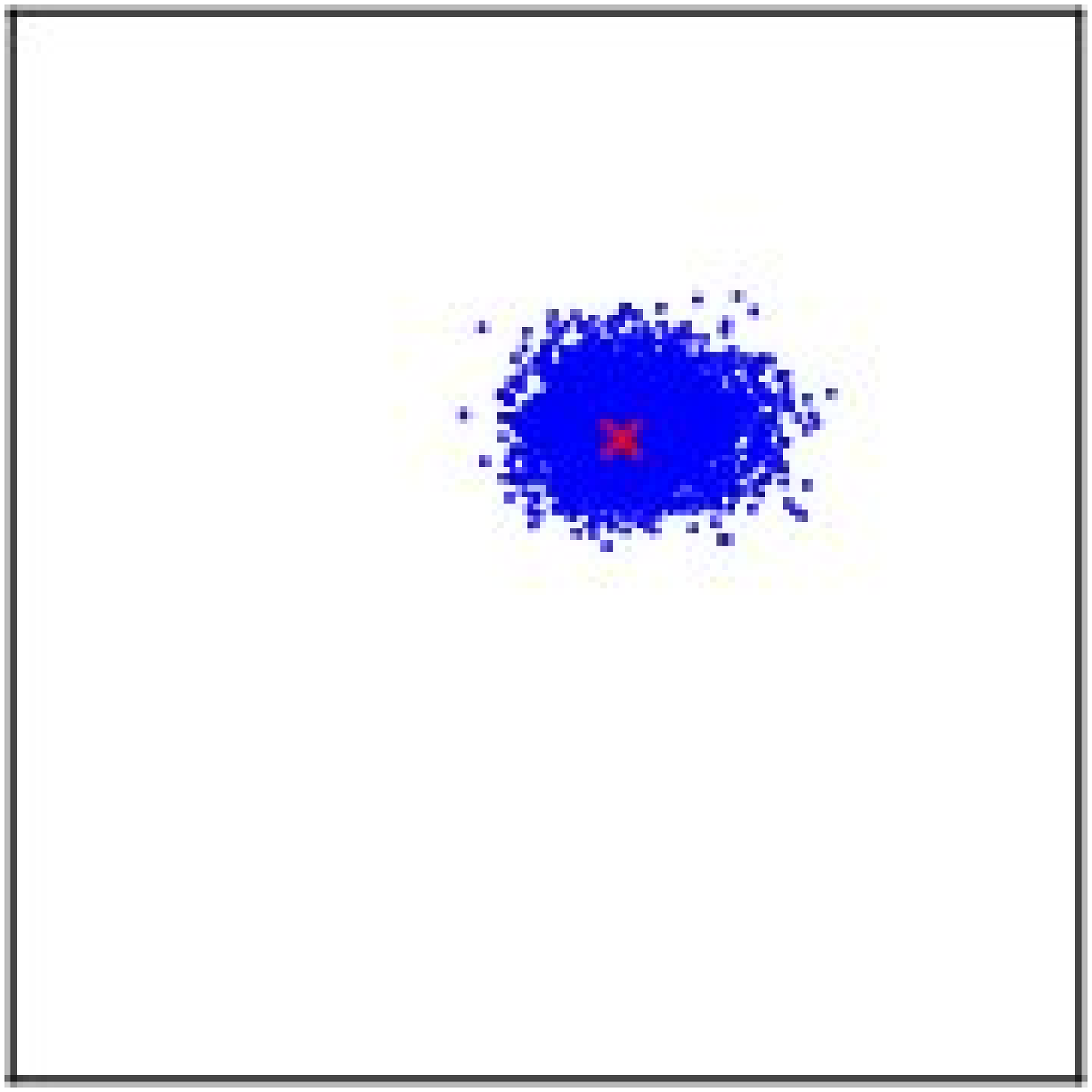}}
  \subfigure[
  $n=1000$
  ]{
    \label{fig:csamples_1000}
    \includegraphics[type=eps, ext=.eps, read=.eps, width=0.4\textwidth]{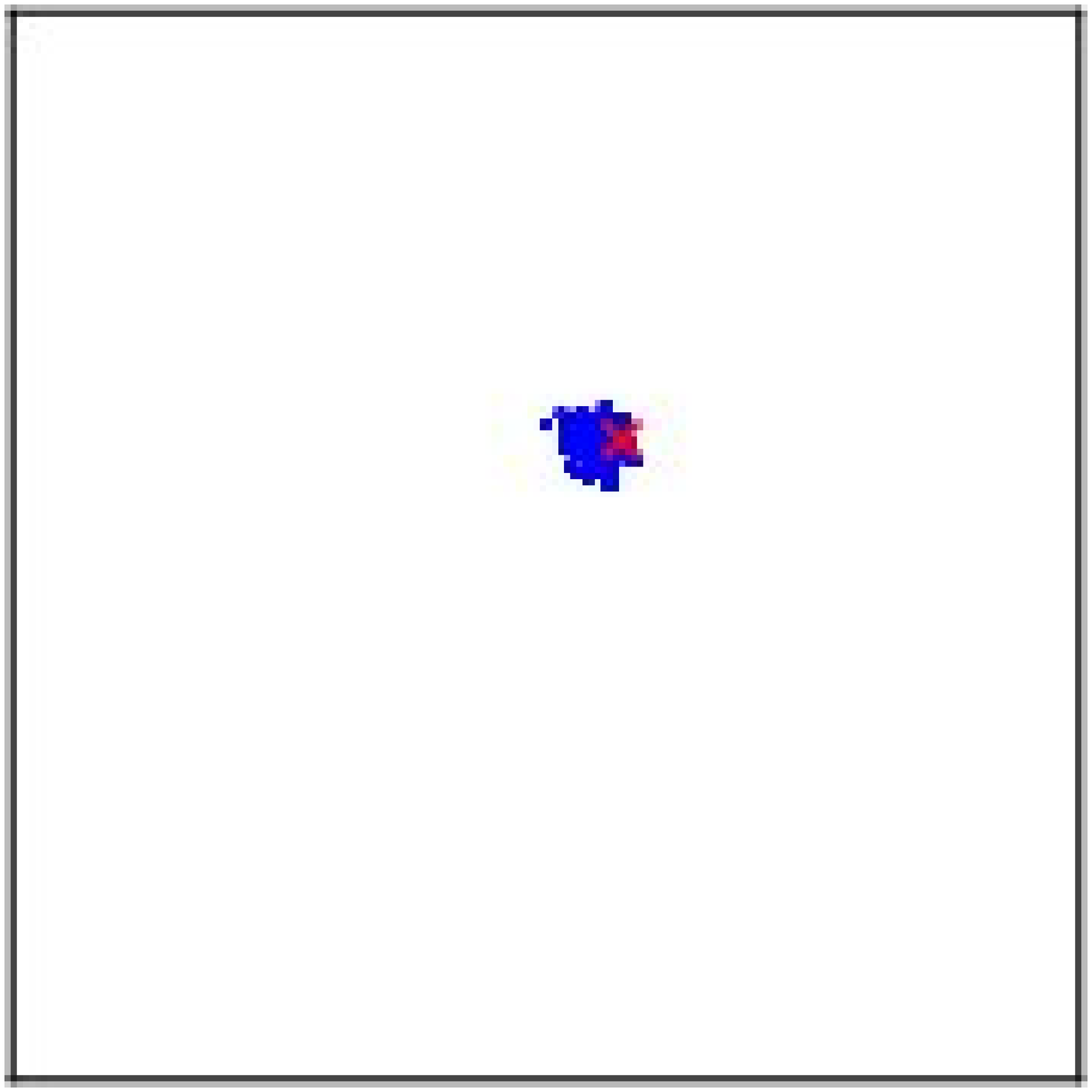}}
  \caption{
  The marginal distribution of $\mathbb{P}(c,v_0 |Y_n)$ with respect to $c$ are depicted on the
square $1.4\times10^{-4}$ by $\times 10^{-4}$.
The red cross marks the true wave velocity $c=(-0.5,-1.0)$.
  }
  \label{fig:c-samples}
\end{figure}

\section{Conclusions}
In this paper,
we study an infinite dimensional state estimation problem in
the presence of model error.
For the statistical model of advection equation on a torus,
with noisily observed functions in discrete time,
the large data limit of the filter and the smoother
both recover the truth in the perfect model scenario. 
If the actual wave velocity differs from the
true wave velocity in a time-integrable fashion then
the filter recovers the truth, but the smoother is in error
by a constant phase shift, determined by the integral
of the difference in wave velocities.
When the difference in wave velocities is constant neither
filtering nor smoothing recovers the truth in the
large data limit. 
And when the difference in wave velocities is
a fluctuating random field, however small, neither
filtering nor smoothing recovers the truth in the
large data limit.

In this paper we consider the dynamics as a hard
constraint, and do not allow for the addition of
mean zero Gaussian noise to the time evolution
of the state. Adding such noise to the model is
sometimes known as a weak constraint
approach in the data assimilation community and
the relative merits of hard and weak constraint
approaches are widely debated; see \cite{bennett2002inverse, apte2008data} for
discussion and references. New techniques of
analysis would be required to study the weakly
constrained problem, because the inverse covariance
does not evolve linearly as it does for the hard
constraint problem we study here. We leave this for
future study.

There are a number of other ways in which the analysis in this
paper could be generalized, in order to obtain
a deeper understanding of filtering methods for high
dimensional systems. 
These include: (i) the study of dissipative model dynamics;
(ii) the study of nonlinear wave propagation problems;
(iii) the study of Lagrangian rather than Eulerian data.
Many other generalizations are also possible.  
For nonlinear systems, the key computational
challenge is to find filters which can be justified,
either numerically or analytically, and which are
computationally feasible to implement. There is already
significant activity in this direction, and studying the
effect of model/data mismatch will form an important
part of the evaluation of these methods. 

\ack The authors would like to thank the following institutions for
financial support: NERC, EPSRC, ERC and ONR.
The authors also thank The Mathematics Institute and Centre for Scientific 
Computing at Warwick University for supplying valuable computation time.

\appendix
\renewcommand{\theappxlem}{A.\arabic{appxlem}}
\section{Basic Theorems on Gaussian Measures}
\label{app:A}
Suppose the probability measure $\mu$ is defined on the Hilbert space $\mathcal{H}$.
A function $m \in \mathcal{H}$ is called the mean of $\mu$ if,
for all $\ell$ in the dual space of linear functionals on $\mathcal{H}$,
\begin{eqnarray*}
\ell(m)= \int_{\mathcal{H}}\ell(x) \mu(dx),
\end{eqnarray*}
and a linear operator $\mathcal{C}$ is called 
the covariance operator if
for all $k, \ell$ in the dual space of $\mathcal{H}$,
\begin{eqnarray*}
k\left( \mathcal{C} \ell \right)= \int_{\mathcal{H}} k(x-m)\ell(x-m) \mu(dx).
\end{eqnarray*}
In particular,
a measure $\mu$ is called Gaussian if
$\mu \circ \ell^{-1} = \mathcal{N}\left( m_\ell, \sigma_\ell^2 \right)$
for some $m_\ell, \sigma_\ell \in \mathbb{R}$.
Since the mean and covariance operator completely determine a Gaussian measure,
we denote a Gaussian measure with mean $m$ and covariance operator $\mathcal{C}$
by $\mathcal{N}(m,\mathcal{C})$.

The following lemmas, all of which can be found in~\cite{Prato92},
summarize the properties of Gaussian measures which we require for this paper.

\begin{appxlem}
\label{gaussiancovop}
If $\mathcal{N}(0, \mathcal{C})$ is a Gaussian measure on a Hilbert space $\mathcal{H}$,
then $\mathcal{C}$ is a self-adjoint, positive semi-definite nuclear operator on $\mathcal{H}$.  
Conversely, if $m \in \mathcal{H}$ and $\mathcal{C}$ is a self-adjoint, positive semi-definite,
nuclear operator on $\mathcal{H}$, then there is a Gaussian measure
$\mu = \mathcal{N}(m,\mathcal{C})$ on $\mathcal{H}$.  
\end{appxlem}

\begin{appxlem}
\label{condgaussian}
Let $\mathcal{H} = \mathcal{H}_1 \oplus \mathcal{H}_2$ be a separable Hilbert space
with projections $\Pi_i: \mathcal{H} \to \mathcal{H}_i, i=1,2$.
For an $\mathcal{H}$-valued Gaussian random variable $(x_1,x_2)$ with mean
$m = (m_1,m_2)$ and positive-definite covariance operator $\mathcal{C}$,
denote $C_{ij} = \Pi_i \mathcal{C} \Pi_j^*$.
Then the conditional distribution of $x_1$ given $x_2$ is Gaussian with mean
\begin{eqnarray*}
m_{1|2} = m_1 - {C}_{12}{C}_{22}^{-1}(m_2 - x_2),
\end{eqnarray*}
and covariance operator
\begin{eqnarray*}
\mathcal{C}_{1|2} = {C}_{11} - {C}_{12}{C}_{22}^{-1}{C}_{21}.
\end{eqnarray*}
\end{appxlem}

\begin{appxlem}
\emph{(Feldman-Hajek)}
\label{FH}
Two Gaussian measures $\mu_i = \mathcal{N}(m_i, \mathcal{C}_i)$, $i = 1, 2$, on a Hilbert
space $\mathcal{H}$ are either singular or equivalent.  They are equivalent if and only if the
following three conditions hold:\\
(i) $\mbox{Im}\left(\mathcal{C}_1^{\frac{1}{2}}\right)
=\mbox{Im}\left(\mathcal{C}_2^{\frac{1}{2}}\right):=E$;\\
(ii) $m_1-m_2 \in E$;\\
(iii) the operator $T:=\left(\mathcal{C}_1^{-\frac{1}{2}}\mathcal{C}_2^{\frac{1}{2}}\right)
\left(\mathcal{C}_1^{-\frac{1}{2}}\mathcal{C}_2^{\frac{1}{2}}\right)^*-I$
is Hilbert-Schmidt in $\bar{E}$.  
\end{appxlem}

\begin{appxlem}
\label{FH1}
For any two positive-definite, self-adjoint operators
$\mathcal{C}_i$, $i=1,2$, on a Hilbert space $\mathcal{H}$,
the condition
$\mbox{Im}\left(\mathcal{C}_1^{\frac{1}{2}}\right) \subseteq
\mbox{Im}\left(\mathcal{C}_2^{\frac{1}{2}}\right)$ holds if and only if
there exists a constant $K > 0$ such that 
\begin{eqnarray*}
\left( h, \mathcal{C}_1 h \right) \leq K \left( h, \mathcal{C}_2 h \right),
\quad \forall h \in \mathcal{H}
\end{eqnarray*}
where $( \cdot, \cdot )$ denotes the inner product on $\mathcal{H}$.
\end{appxlem}

\begin{appxlem}
\label{gaussianmoment}
For Gaussian $X$ on a Hilbert space $\mathcal{H}$ with norm $\left\| \cdot \right\|$
and for any integer $n$,
there is constant $C_n \geq 0$ such that
$\mathbb{E} ( \left\| X \right\|^{2n} )
\leq C_n \left( \mathbb{E} \left( \left\| X \right\|^{2} \right) \right)^n$.
\end{appxlem}

\section{Proof of Limit Theorems}
\label{app:B}
In this Appendix, we will prove
the Limit Theorems~\ref{limitthm1},
\ref{limitthm2},
\ref{limitthm3},
where
$\mathcal{L} = \mathcal{L}'$,
$\mathcal{L} \neq \mathcal{L}'$,
$\mathcal{L}(t) \neq \mathcal{L}'(t)$,
and
Theorem \ref{limitthm4}
where
$\mathcal{L}(t) \neq \mathcal{L}'(t;\omega')$,
respectively.
In all cases, we use the notations
$e^{-t\mathcal{L}}$ and
$e^{-t\mathcal{L}'}$ to denote the forward solution operators through $t$ time units
(from time zero in the non-autonomous case).
We denote by ${M}$ the putative limit for $m'_n$. The identity
\begin{eqnarray}
\label{eq:identity}
\left( nI+\Gamma \mathcal{C}_{0}^{-1}\right) (m'_{n}-{M})
& = \Gamma \mathcal{C}_{0}^{-1} \left( m_{0}-{M} \right)
+\sum_{l=0}^{n-1} e^{t_{l+1} \mathcal{L} }\eta'_{l+1} \nonumber\\
& \quad +\sum_{l=0}^{n-1} \left( e^{ t_{l+1}
\left(\mathcal{L} -\mathcal{L}' \right)}u -{M} \right),
\end{eqnarray}
obtained from Equation~(\ref{eq:posmean}),
will be used to show $m'_n \to {M}$ .
In Equation~(\ref{eq:identity}),
we will choose $M$ so that
the contribution of the last term is asymptotically negligible.
Define the Fourier representations
\begin{eqnarray*}
\mathbf{e}_n & \equiv  m'_{n}-{M} = \sum_{k\in\mathbb{K}} \hat{\mathbf{e}}_{n}(k) \phi_k, \\
\mathbf{\xi}_l & \equiv  e^{t_{l+1} \left(\mathcal{L}-\mathcal{L}'\right)}u
-{M} = \sum_{k\in\mathbb{K}} \hat{\mathbf{\xi}}_{l}(k) \phi_k.
\end{eqnarray*}
Then $M$ will be any one of
$u$, $\mathcal{F}_{(p,q)}u$, $\langle u\rangle$ and $u_\alpha.$
Hence there is $C_1$ independent of $k,l$ such that 
$\hat{\mathbf{\xi}}_l(0)=0$ and
$\mathbb{E}\vert \hat{\mathbf{\xi}}_l(k) \vert \leq C_1\vert \left(u,\phi_k \right)\vert$
with $C_1<\infty$ (the expectation here is trivial
except in the case of random ${\mathcal L}'$).
Using Equation~(\ref{eq:newKLexpn}),
these Fourier coefficients satisfy the relation
\begin{eqnarray*}
\left(n+\frac{\gamma_k}{\lambda_k} \right) \hat{\mathbf{e}}_n (k)
& = \frac{\gamma_k}{\lambda_k} \hat{\mathbf{e}}_0 (k)
+\sqrt{\gamma'_k}\sum_{l=0}^{n-1} (\mathbf{g}_k')^{l+1}
+\sum_{l=0}^{n-1} \hat{\mathbf{\xi}}_l(k),
\quad k\in \mathbb{K}.
\end{eqnarray*}

In order to prove $m'_n \to {M}$ in $L^2(\Omega';H^s(\mathbf{T}^2))$,
we use the monotone convergence theorem to obtain the following inequalities,

\begin{eqnarray}
\label{eq:coreinequality}
\fl n^\delta \left\| \mathbf{e}_n \right\|^2_{L^2(\Omega';H^s(\mathbf{T}^2))}
&= n^\delta \mathbb{E} \left\| \mathbf{e}_n \right\|^2_{H^s(\mathbf{T}^2)} \nonumber\\
&= \sum_{k \in \mathbb{K}^+} |k|^{2s} n^\delta \mathbb{E} \lvert \hat{\mathbf{e}}_n(k) \rvert^2
+n^\delta \mathbb{E} \lvert \hat{\mathbf{e}}_n(0) \rvert^2 \nonumber\\
&= \sum_{k \in \mathbb{K}^+} |k|^{2s}
\frac{n^\delta}{\left(n+{\gamma_k}/{\lambda_k} \right)^2}
\Bigg[  \left(\frac{\gamma_k}{\lambda_k}\right)^2 |\hat{\mathbf{e}}_0 (k)|^2 +\gamma_k' n \nonumber\\
& \quad +2 \mbox{Re} \left\lbrace \frac{\gamma_k}{\lambda_k}
  \bar{\hat{\mathbf{e}}}_0(k) \mathbb{E} \left( \sum_{l=0}^{n-1}
    \hat{\mathbf{\xi}}_l(k) \right) \right\rbrace
+\mathbb{E} \left\lvert \sum_{l=0}^{n-1} \hat{\mathbf{\xi}}_l(k) \right\rvert^2 \Bigg] \nonumber\\
& \quad + \frac{n^\delta}{(n+{\gamma_0}/{\lambda_0})^2}
\left[ 
  \left( \frac{\gamma_0}{\lambda_0} \right)^2 \lvert \hat{\mathbf{e}}_0(0) \rvert^2 +\gamma_0' n
\right] \nonumber\\
& \leq \sum_{k \in \mathbb{K}^+} |k|^{2s}
{n^{\delta-2}} \Bigg[ \left(\frac{\gamma_k}{\lambda_k}\right)^2
|\hat{\mathbf{e}}_0 (k)|^2 \nonumber\\
& \quad +\gamma_k' n +2C_1 \frac{\gamma_k}{\lambda_k}
\lvert \hat{\mathbf{e}}_0 (k) \rvert \lvert (u,\phi_k)\rvert n
+\mathbb{E} \left\lvert \sum_{l=0}^{n-1} \hat{\mathbf{\xi}}_l(k) \right\rvert^2 \Bigg] \nonumber\\
& \quad
+n^{\delta-2} \left[ \left( \frac{\gamma_0}{\lambda_0} \right)^2
  \lvert \hat{\mathbf{e}}_0(0) \rvert^2 + \gamma_0' n\right] \nonumber\\
 & \leq \sum_{k \in \mathbb{K}^+} |k|^{2s}
 {n^{\delta-2}} \Bigg[ \left(\frac{\gamma_k}{\lambda_k}\right)^2
 |\hat{\mathbf{e}}_0 (k)|^2 +\gamma_k' n  \nonumber\\
 & \quad
 +C_1\left( \left(\frac{\gamma_k}{\lambda_k}\right)^2 |\hat{\mathbf{e}}_0 (k)|^2 
 + \lvert (u,\phi_k)\rvert^2 \right) n + \mathbb{E} \left\lvert \sum_{l=0}^{n-1}
\hat{\mathbf{\xi}}_l(k) \right\rvert^2 \Bigg] \nonumber\\
& \quad
 +n^{\delta-2} \left[ \left( \frac{\gamma_0}{\lambda_0} \right)^2
 \lvert \hat{\mathbf{e}}_0(0) \rvert^2 + \gamma_0' n\right] \nonumber\\
 & \leq \left(C \left\| m_0 -{M} \right\|_{H^{s+\kappa}(\mathbf{T}^2)}\right)n^{\delta-2}
 +\left( \sum_{k \in \mathbb{K}} |k|^{2s}\gamma_k' \right) n^{\delta-1} \nonumber\\
 & \quad + \left( C_1  \left\| u \right\|_{H^{s}(\mathbf{T}^2)} \right) n^{\delta -1}
 + \left( \sum_{k \in \mathbb{K}^+} |k|^{2s}  n^{\delta-2}
 \mathbb{E} \left\lvert \sum_{l=0}^{n-1} \hat{\mathbf{\xi}}_l(k) \right\rvert^2 \right) \nonumber\\
 & \leq C\left( \left\| m_0\right\|_{H^{s+\kappa}(\mathbf{T}^2)}
 +\left\| u \right\|_{H^{s+\kappa}(\mathbf{T}^2)}\right)n^{\delta-2}\nonumber\\
 & \quad + \left( \sum_{k \in \mathbb{K}} |k|^{2s}\gamma_k'
 +C_1  \left\| u \right\|_{H^{s}(\mathbf{T}^2)} \right) n^{\delta-1} \nonumber\\
&\quad
 + \left( \sum_{k \in \mathbb{K}} |k|^{2s}  n^{\delta-2}
 \mathbb{E} \left\lvert \sum_{l=0}^{n-1} \hat{\mathbf{\xi}}_l(k) \right\rvert^2 \right),
\end{eqnarray}

and here the first two terms in the last equation
can be controlled by Assumptions~\ref{ass1}.
In order to find $\delta \in [0,1]$ such that this equation is $O(1)$ or $o(1)$,
\begin{eqnarray}
\label{eq:coreterm}
\sum_{k \in \mathbb{K}} |k|^{2s}  n^{\delta-2}
\mathbb{E} \left\lvert \sum_{l=0}^{n-1} \hat{\mathbf{\xi}}_l(k) \right\rvert^2
\end{eqnarray}
is the key term. This term arises from the model error, i.e., the discrepancy between the operator
$\mathcal{L}$ used in the statistical model and the operator $\mathcal{L}'$ which generates
the data.
We analyze it, in various cases, in the subsections
which follow.

In order to prove $m'_n \to {M}$ in $H^s(\mathbf{T}^2), \; \Omega'-a.s.$,
suppose we have
\begin{eqnarray}
\label{eq:coreterm2}
\frac{1}{n}\sum_{l=0}^{n-1} \hat{\mathbf{\xi}}_l(k)
\to 0 \quad \Omega'-a.s.,
\end{eqnarray}
for each $k \in \mathbb{K}$.
We then use the strong law of large numbers to obtain the following inequalities,
which holds $\Omega'-a.s.$,
\begin{eqnarray}
\label{eq:coreinequality2}
\fl \left\| \mathbf{e}_n \right\|^2_{H^s(\mathbf{T}^2)}
& = \sum_{k \in \mathbb{K}^+} |k|^{2s} \lvert \hat{\mathbf{e}}_n(k) \rvert^2
+\lvert \hat{\mathbf{e}}_n(0) \rvert^2 \nonumber\\
& = \sum_{k \in \mathbb{K}^+} |k|^{2s}
\left\lvert \frac{
\left({\gamma_k}/{\lambda_k}\right) \hat{\mathbf{e}}_0 (k)
+\sqrt{\gamma'_k}\sum_{l=0}^{n-1} (\mathbf{g}_k')^{l+1}
+\sum_{l=0}^{n-1} \hat{\mathbf{\xi}}_l(k)
}{n+{\gamma_k}/{\lambda_k}}
\right\rvert^2 \nonumber\\
& \quad
+\left\lvert
\frac{
\left({\gamma_0}/{\lambda_0}\right) \hat{\mathbf{e}}_0 (0)
+\sqrt{\gamma'_0}\sum_{l=0}^{n-1} (\mathbf{g}_0')^{l+1}
}{
n+{\gamma_0}/{\lambda_0}
}
\right\rvert^2 \nonumber\\
& \leq 
C \sum_{k \in \mathbb{K}} (1+|k|^{2s})
\Bigg(
\left\lvert
\frac{
\left({\gamma_k}/{\lambda_k}\right) \hat{\mathbf{e}}_0 (k)
}{
n
}
\right\rvert^2
+\left\lvert
\frac{
\sqrt{\gamma'_k}\sum_{l=0}^{n-1} (\mathbf{g}_k')^{l+1}
}{
n
}
\right\rvert^2 \nonumber\\
& \quad +\left\lvert
\frac{
\sum_{l=0}^{n-1} \hat{\mathbf{\xi}}_l(k)
}{
n
}
\right\rvert^2
\Bigg) \nonumber\\
& \leq 
C \sum_{k \in \mathbb{K}} (1+|k|^{2s}) 
\left( |k|^{2\kappa} \lvert \hat{\mathbf{e}}_0 (k) \rvert^2 +\gamma_k'+\lvert (u,\phi_k) \rvert^2
\right) \nonumber\\
& \leq C\left( \left\| m_0\right\|_{H^{s+\kappa}(\mathbf{T}^2)}
+\left\| u \right\|_{H^{s+\kappa}(\mathbf{T}^2)}
+ \sum_{k \in \mathbb{K}} |k|^{2s}\gamma_k'
+ \left\| u \right\|_{H^{s}(\mathbf{T}^2)}\right).
\end{eqnarray}
Therefore, using Weierstrass M-test,
we have
$\left\| \mathbf{e}_n \right\|^2_{H^s(\mathbf{T}^2)} \to 0, \; \Omega'-a.s.$,
once Equation~(\ref{eq:coreterm2}) is satisfied.

\subsection{Proof of Theorem~\ref{limitthm1}}
This proof is given directly after the theorem statement.
For completeness we note that Equation~(\ref{eq:discov-a}) follows from
Equation~(\ref{eq:coreinequality}) with $\delta =1$.
Once it has been established that 
\begin{eqnarray*}
\left\| \cdot \right\|_{L^2(\Omega';H^s(\mathbf{T}^2))} = O\left(n^{-\delta/2}\right),
\end{eqnarray*}
then the proof that
\begin{eqnarray*}
\left\| \cdot \right\|_{H^s(\mathbf{T}^2)} = o\left(n^{-\theta}\right)
\quad \Omega'-a.s.,
\end{eqnarray*}
for any $\theta < \delta/2$ follows from a Borel-Cantelli argument as shown in the proof of
Theorem~\ref{limitthm1}.
We will not repeat this argument for the
proofs of Theorems~\ref{limitthm2},
\ref{limitthm3} and \ref{limitthm4}.

\subsection{Proof of Theorem~\ref{limitthm2}}
\subsubsection{}
When $\mathcal{L} \neq \mathcal{L}'$ and $\Delta t \, \delta c=(p'/p,q'/q)$,
choose ${M}=\mathcal{F}_{(p,q)}u$ then
\begin{eqnarray*}
\hat{\mathbf{\xi}}_l(k)
& = e^{2\pi i (k \cdot \delta c) t_{l+1}} (u,\phi_k)
- \delta^k_{(p,q)} (u,\phi_k) \\
& = \left(1-\delta^k_{(p,q)} \right)e^{2\pi i (k \cdot \delta c) t_{l+1}} (u,\phi_k),
\end{eqnarray*}
where $\delta^{k}_{(p,q)}$ is $1$ if
$k_1$ and $k_2$ for $k=(k_1,k_2)$ are multiples of $p$ and $q$ respectively,
and $0$ otherwise.
Using
\begin{eqnarray*}
\left\lvert \sum_{l=0}^{n-1} \hat{\mathbf{\xi}}_l(k) \right\rvert^2
& = \left(1-\delta^k_{(p,q)}\right)^2
\left\lvert \sum_{l=0}^{n-1} e^{2\pi i (k\cdot \delta c)t_{l+1}} \right\rvert^2 |(u,\phi_k)|^2 \\
& = \left(1-\delta^k_{(p,q)}\right) \left[ \frac{\sin(n \pi (k \cdot \delta c) \Delta t)}
{\sin(\pi (k \cdot \delta c) \Delta t)}\right]^2 |(u,\phi_k)|^2 \\
& \leq \left[\sin^2\left( \pi\left(\frac{1}{p}+\frac{1}{q} \right) \right)
\right]^{-1}|(u,\phi_k)|^2,
\end{eqnarray*}
the quantity in Equation~(\ref{eq:coreterm}) becomes
\begin{eqnarray*}
\left(
\left[\sin^2\left( \pi\left( \frac{1}{p}+\frac{1}{q} \right) \right) \right]^{-1}
\left\| u \right\|_{H^{s}(\mathbf{T}^2)}
\right) n^{\delta-2},
\end{eqnarray*}
so that from Equation~(\ref{eq:coreinequality})
\begin{eqnarray*}
n \left\| \mathbf{e}_n \right\|^2_{L^2(\Omega';H^s(\mathbf{T}^2))}
=O(1).
\end{eqnarray*}

\subsubsection{}
\label{B22}
When $\mathcal{L} \neq \mathcal{L}'$ and $\Delta t \, \delta c \in
\mathbb{R}\backslash\mathbb{Q} \times \mathbb{R}\backslash\mathbb{Q}$,
choose ${M}=\langle u \rangle$ then 
\begin{eqnarray*}
\hat{\mathbf{\xi}}_l(k) = 
e^{2\pi i (k \cdot \delta c) t_{l+1}}(u,\phi_k),
\end{eqnarray*}
for $k \in \mathbb{K}^+$.
It is immediate that 
\begin{eqnarray*}
\left\| \mathbf{e}_n \right\|_{H^s(\mathbf{T}^2)}
=o(1)
\quad \Omega'-a.s.,
\end{eqnarray*}
since
\begin{eqnarray*}
\frac{1}{n} \sum_{l=0}^{n-1}e^{2\pi i (k\cdot \delta c) t_{l+1}}=o(1),
\end{eqnarray*}
as $n \to \infty$,
by ergodicity ~\cite{Breiman} and Equation~(\ref{eq:coreterm2}).
Furthermore, when $\delta =0$, the quantity in
Equation~(\ref{eq:coreterm}) is bounded by 
\begin{eqnarray*}
\sum_{k \in \mathbb{K}}
|k|^{2s}\left\lvert
\frac{1}{n}
\sum_{l=0}^{n-1}e^{2\pi i (k\cdot \delta c) t_{l+1} }
\right\rvert^2 |(u,\phi_k)|^2
\leq C \| u\|_{H^s(\mathbf{T}^2)},
\end{eqnarray*}
and Weierstrass M-test can be used to show
\begin{eqnarray*}
\left\| \mathbf{e}_n \right\|^2_{L^2(\Omega';H^s(\mathbf{T}^2))} = o(1).
\end{eqnarray*}

\subsection{Proof of Theorem~\ref{limitthm3}}
When $\mathcal{L}(t) \neq \mathcal{L}'(t)$ and
$\int^t_0 \delta c(s)\,ds =\alpha+O(t^{-\beta})$,
choose ${M} = u_\alpha$ then
\begin{eqnarray*}
\hat{\mathbf{\xi}}_l(k) = 
\left(
e^{2\pi i k \cdot \left( \int^{t_{l+1}}_0 \delta c(s)\,ds \right)}
- e^{2\pi i k \cdot \alpha} \right)(u,\phi_k),
\end{eqnarray*}
and we obtain
\begin{eqnarray*}
\left\lvert \sum_{l=0}^{n-1} \hat{\mathbf{\xi}}_l(k) \right\rvert^2
=\left\lvert \sum_{l=0}^{n-1} \left(
e^{2\pi i k \cdot \left(\int^{t_{l+1}}_0 \delta c(s)\,ds -\alpha\right)}
- 1 \right)
\right\rvert^2 |(u,\phi_k)|^2.
\end{eqnarray*}
Now utilizing that 
\begin{eqnarray*}
\left\lvert \sum_{l=0}^{n-1} \left( e^{ i x_l } -1 \right) \right\rvert
\leq C \sum_{l=0}^{n-1} \lvert x_l \rvert \leq C_2 n^{1-\beta},
\end{eqnarray*}
when $x_l =O\left(l^{-\beta}\right)$ and for $\beta \in (0, 1/2]$,
Equation~(\ref{eq:coreinequality}) gives 
\begin{eqnarray*}
n^\delta \left\| \mathbf{e}_n \right\|^2_{L^2(\Omega';H^s(\mathbf{T}^2))}
& \leq C\left( \left\| m_0\right\|_{H^{s+\kappa}(\mathbf{T}^2)}
+\left\| u \right\|_{H^{s+\kappa}(\mathbf{T}^2)}\right)n^{\delta-2} \\
& \quad +\left( \sum_{k \in \mathbb{K}} |k|^{2s}\gamma_k'
+C_1  \left\| u \right\|_{H^{s}(\mathbf{T}^2)} \right) n^{\delta-1}\\
& \quad + \left( C_2 \left\| u \right\|^2_{H^{s}(\mathbf{T}^2)} \right) n^{\delta - 2\beta}
=O(1),
\end{eqnarray*}
for $\delta=1 \wedge 2\beta$.

\subsection{Proof of Theorem~\ref{limitthm4}}
When $\mathcal{L}(t) \neq \mathcal{L}'(t;\omega')$
and
$\int^t_0 c'(s;\omega')\,ds = \int^t_0 c(s)\,ds - \varepsilon {W}(t)$,
choose ${M}=\langle u\rangle$
then we obtain
\begin{eqnarray*}
{\mathbf{\xi}}_l = u\left( \cdot + \varepsilon W(t_{l+1}) \right) -\langle u \rangle,
\end{eqnarray*}
and
\begin{eqnarray*}
\hat{\mathbf{\xi}}_l(k) = 
e^{2\pi i k\cdot \varepsilon W(t_{l+1})}(u,\phi_k),
\end{eqnarray*}
for $k \in \mathbb{K}^+$.  
Using
\begin{eqnarray*}
\label{eq:coreineqality2}
\fl \mathbb{E} \left\lvert \sum_{l=0}^{n-1} \hat{\mathbf{\xi}}_l(k) \right\rvert^2
&= \mathbb{E} \left( \sum_{l=0}^{n-1}\sum_{l'=0}^{n-1}
e^{2\pi i k \cdot \varepsilon \left( W(t_{l+1})-W(t_{l'+1}) \right) }| (u,\phi_k)|^2 \right)\\
&= \sum_{l=0}^{n-1}\sum_{l'=0}^{n-1} e^{-2\pi^2 \varepsilon^2 |k|^2 \Delta t |l-l'|}| (u,\phi_k)|^2 \\
&= \left( n+ 2\sum_{l=1}^{n-1}\sum_{l'=0}^{n-1}
e^{-2\pi^2 \varepsilon^2 |k|^2 \Delta t |l-l'|}| \right) (u,\phi_k)|^2 \\
&= \left[n+\frac{2}{e^{2\pi^2 \varepsilon^2 |k|^2 \Delta t}-1}
\left( \frac{e^{-2\pi^2\varepsilon^2|k|^2 \Delta t(n-1)}-1}
{e^{2\pi^2\varepsilon^2|k|^2 \Delta t}-1} + n-1 \right) \right]|(u,\phi_k)|^2 \\
&\leq  \left[n+\frac{2}{e^{2\pi^2\varepsilon^2\Delta t}-1}n \right]|(u,\phi_k)|^2
=  \left(\frac{e^{2\pi^2\varepsilon^2\Delta t}+1}{e^{2\pi^2\varepsilon^2\Delta t}-1}\right) n
|(u,\phi_k)|^2,
\end{eqnarray*}
we get
\begin{eqnarray*}
\sum_{k \in \mathbb{K}^+}|k|^{2s}n^{\delta-2}
\mathbb{E} \left\lvert \sum_{l=0}^{n-1} \hat{\mathbf{\xi}}_l(k) \right\rvert^2
\leq  \left( \frac{e^{2\pi^2\varepsilon^2\Delta t}+1}{e^{2\pi^2\varepsilon^2\Delta t}-1} \right)
\left\| u \right\|^2_{H^s(\mathbf{T}^2)} n^{\delta-1}
=O(1),
\end{eqnarray*}
for $\delta = 1$.

\section*{References}
\bibliographystyle{iopart-num}
\bibliography{bibliographyfile}

\end{document}